\documentclass[final]{siamltex}

\usepackage{pstricks}
\usepackage{graphicx}

\usepackage{psfrag}
\usepackage{fancybox}
\usepackage{amsfonts,amssymb}
\usepackage{float,longtable}
\usepackage[centertags]{amsmath}
\usepackage{graphicx}
\usepackage[latin1]{inputenc}
\usepackage{newlfont}
 \begingroup
 \newtheorem{thm}{Theorem}[section]
 \newtheorem{lem}[thm]{Lemma}
 \newtheorem{prop}[thm]{Proposition}
 \newtheorem{cor}[thm]{Corollary}
 \endgroup

 \begingroup
 \newtheorem{defn}[thm]{Definition}
 
 \endgroup

 \begingroup
 \newtheorem{rem}[thm]{Remark}
 \endgroup

 \numberwithin{equation}{section}

\title{Aubry-Mather Measures in the Non Convex setting\thanks{F. Cagnetti was partially supported by FCT 
through the CMU$|$Portugal and UTAustin$|$Portugal programs. 
D. Gomes was partially supported by CAMGSD-LARSys through FCT Program POCTI-FEDER 
and by grants PTDC/MAT/114397/2009,
UTAustin/MAT/0057/2008, and UTA-CMU/MAT/0007/2009.
H.~V. Tran was partially supported by VEF fellowship 
and NSF Grant DMS--0903201.}}

\author{F. Cagnetti\footnotemark[2]\ 
\and D.\ Gomes\footnotemark[2]\ 
\and H.~V. Tran\footnotemark[3]}

\begin{document}
\maketitle

\renewcommand{\thefootnote}{\fnsymbol{footnote}}

\footnotetext[2]{Departamento de Matem\'atica, Instituto Superior T\'ecnico,
 Av. Rovisco Pais, 1049-001, Lisbon. (cagnetti@math.ist.utl.pt \& dgomes@math.ist.utl.pt).}
\footnotetext[3]{Department of Mathematics,
University of California Berkeley, CA, 94720-3840. (tvhung@math.berkeley.edu).}

\renewcommand{\thefootnote}{\arabic{footnote}}

\begin{abstract}
The adjoint method, introduced in \cite{Ev} and \cite{Tr}, is used 
to construct  analogs to the Aubry-Mather measures for non convex Hamiltonians.
More precisely, a general construction of probability measures, 
that in the convex setting agree with Mather measures, is provided.
These measures may fail to be invariant under the Hamiltonian flow 
and a dissipation arises, which is described by a positive semi-definite matrix of Borel measures. 
However, in the case of uniformly quasiconvex Hamiltonians the dissipation vanishes, 
and as a consequence the invariance is guaranteed.
\end{abstract}




\begin{section}{Introduction} 

Let us consider a periodic Hamiltonian system whose energy is described by a smooth Hamiltonian 
$H: \mathbb{T}^n \times \mathbb{R}^n \to \mathbb{R}$.
Here $\mathbb{T}^n$ denotes the $n$-dimensional torus, $n \in \mathbb{N}$. 
It is well known that the time evolution $t \mapsto (\mathbf{x} (t), \mathbf{p} (t))$ 
of the system is obtained by solving the Hamilton's ODE
\begin{equation} \label{HE}
\begin{cases}
\vspace{.2cm}
\dot{\mathbf{x}} = - D_p H ( \mathbf{x} , \mathbf{p} ) , \\
\dot{\mathbf{p}} = D_x H ( \mathbf{x} , \mathbf{p} ).
\end{cases}
\end{equation}
Assume now that, for each $P \in \mathbb{R}^n$, there exists a constant $\overline{H} (P)$
and a periodic function $u (\cdot,P)$ solving the following time independent Hamilton-Jacobi equation
\begin{equation} \label{statHJ}
H ( x , P + D_x u (x,P)) = \overline{H} (P).
\end{equation}
Suppose, in addition, that both $u (x,P)$ and $\overline{H} (P)$ are smooth functions.
Then, if the following relations
\begin{equation} \label{chvar}
X = x + D_P u ( x, P ), \quad \quad \quad \quad 
p = P + D_x u ( x, P ),
\end{equation}
define a smooth change of coordinates $X (x,p)$ and $P (x,p)$, the ODE \eqref{HE} can be rewritten as 
\begin{equation} \label{easy}
\begin{cases}
\vspace{.2cm}
\dot{\mathbf{X}} = - D_P \overline{H} ( \mathbf{P} ),  \\
\dot{\mathbf{P}} = 0.
\end{cases}
\end{equation}
Since the solution of \eqref{easy} is easily obtained,
solving \eqref{HE} is reduced to inverting the change of coordinates \eqref{chvar}.
Unfortunately, several difficulties arise.

\medskip

Firstly, it is well known that the solutions of the nonlinear PDE \eqref{statHJ} are not smooth in
the general case.
For the convenience of the reader, we recall the definition of viscosity solution.
\begin{defn}
We say that $u$ is a \textit{viscosity solution} of \eqref{statHJ}
if for each $v \in C^{\infty}(\mathbb R^n)$
\begin{itemize}
\item If $u-v$ has a local maximum at a point $x_0 \in \mathbb R^n$ then
$$
H ( x_0 , P + D v (x_0) ) \le \overline{H} (P);
$$
\item If $u-v$ has a local minimum at a point $x_0 \in \mathbb R^n$ then
$$
H ( x_0 , P + D v (x_0) ) \ge \overline{H} (P) .
$$
\end{itemize}
\end{defn}
One can anyway solve \eqref{statHJ} in this weaker sense, as made precise by the following theorem, 
due to Lions, Papanicolaou and Varadhan.
\begin{thm}[See \cite{LPV}] \label{exist}
Let $H: \mathbb{T}^n \times \mathbb{R}^n \to \mathbb{R}$ be smooth such that
\begin{equation} \label{coe}
\lim_{|p| \to + \infty} H (x,p) = + \infty.
\end{equation}
Then, for every $P \in \mathbb{R}^n$ there exists a unique  
$\overline{H} (P) \in \mathbb{R}$ such that \eqref{statHJ} 
admits a $\mathbb{Z}^n$-periodic viscosity solution $u (\cdot,P): \mathbb{T}^n \to \mathbb{R}$.
\end{thm}
We call \eqref{statHJ} the \textit{cell problem}.
It can be proved that all the viscosity solutions of the cell problem are Lipschitz continuous, 
with Lipschitz constants uniformly bounded in $P$.

\medskip

A second important issue is that the solution $u(\cdot,P)$ of \eqref{statHJ} may not be unique, even modulo addition of constants.
Indeed, a simple example is given by the Hamiltonian $H(x,p) = p \cdot (p -~D \psi (x))$, 
where $\psi: \mathbb{T}^n \to \mathbb{R}$ is a smooth fixed function.
In this case, for $P = 0$ and $\overline{H} (0) = 0$, the cell problem is
$$
D u \cdot D (u - \psi) = 0,
$$
which admits both $u \equiv 0$ and $u = \psi$ as solutions.
Therefore, smoothness of $u(x,P)$ in $P$ cannot be guaranteed.

\medskip

Finally, even in the particular case in which both $u (x,P)$ and $\overline{H} (P)$ are smooth, 
relations \eqref{chvar} may not be invertible, or the functions $X (x,p)$ and $P (x,p)$
may not be smooth or globally defined.

\medskip

Therefore, in order to understand the solutions of Hamilton's ODE \eqref{HE} in the general case,
it is very important to exploit the functions $\overline{H} (P)$ and $u (x,P)$, 
and to extract any possible information ``encoded'' in $\overline{H} (P)$ about the dynamics.

\subsection{Classical Results: the convex case}

Classically, the additional hypotheses required in literature on the Hamiltonian $H$ are:

\begin{itemize}

\item[(i)] $H (x,\cdot)$ is strictly convex;

\vspace{.1cm}

\item[(ii)] $H (x,\cdot)$ is superlinear, i.e.
\begin{equation*} 
\lim_{|p| \to + \infty} \frac{H(x,p)}{|p|} = + \infty.
\end{equation*}
\end{itemize}
A typical example is the mechanical Hamiltonian
$$
H(x,p) = \frac{|p|^2}{2} + V (x),
$$
where $V$ is a given smooth $\mathbb{Z}^n$-periodic function.
Also, one restricts the attention to a particular class of  trajectories of \eqref{HE}, 
the so-called one sided \textit{absolute minimizers} of the action integral. 
More precisely, one first defines the Lagrangian $L: \mathbb{T}^n \times \mathbb{R}^n \to \mathbb{R}$ associated to $H$ 
as the Legendre transform of $H$:
\begin{equation} \label{Lagr}
L (x,v): = H^*(x,v) = \sup_{ p \in \mathbb{R}^n} \{ - p \cdot v - H(x,p) \}
\quad \text{ for every }(x,v) \in \mathbb{T}^n \times \mathbb{R}^n.
\end{equation}
Here the signs are set following the Optimal Control convention (see \cite{FS}).
Then, one looks for a Lipschitz curve $\mathbf{x} ( \cdot )$ which minimizes the action integral,
i.e. such that
\begin{equation} \label{min}
\int_0^{T} L ( \mathbf{x} (t) , \mathbf{\dot{x}} (t) ) \, dt
\leq \int_0^{T} L ( \mathbf{y} (t) , \mathbf{\dot{y}} (t) ) \, dt
\end{equation}
for each time $T > 0$ and each Lipschitz curve $\mathbf{y} ( \cdot )$
with $\mathbf{y} ( 0 )=  \mathbf{x} ( 0 )$ and $\mathbf{y} ( T )=  \mathbf{x} ( T )$.
Under fairly general conditions such minimizers exist, 
are smooth, and satisfy the Euler-Lagrange equations
\begin{equation} \label{EL}
\frac{d}{dt} \left[ D_v L ( \textbf{x} (t) , \dot{\textbf{x}} (t) ) \right] = D_x L ( \textbf{x} (t) , \dot{\textbf{x}} (t) ), 
\quad t \in (0,+ \infty).
\end{equation}
It may be shown that if $\textbf{x} (\cdot)$ solves \eqref{min} (and in turn \eqref{EL}), 
then $( \textbf{x} (\cdot) , \textbf{p} (\cdot) )$ is a solution of \eqref{HE},
where $\textbf{p} (\cdot) := - D_v L ( \dot{\textbf{x}} (\cdot) ,\textbf{x} (\cdot))$.
This is a consequence of assumptions (i) and (ii), that in particular guarantee 
a one to one correspondence between Hamiltonian space and Lagrangian space coordinates, 
through the one to one map $\Phi: \mathbb{T}^n \times \mathbb{R}^n \to \mathbb{T}^n \times \mathbb{R}^n$ 
defined as 
\begin{equation} \label{Phi}
\Phi (x,v) := (x, - D_v L (x,v)).
\end{equation}

There are several natural questions related to the trajectories $\mathbf{x} (\cdot)$ 
satisfying \eqref{min}, in particular in what concerns ergodic averages, asymptotic behavior and so on. 
To address such questions it is common to consider the following related problem. 

\medskip

In 1991 John N. Mather (see \cite{Mat}) proposed a relaxed version of \eqref{min}, by considering 
\begin{equation} \label{Mather91}
\min_{ \nu \in \mathcal{D} }
\int_{\mathbb{T}^n \times \mathbb{R}^n} L ( x , v ) \, d \nu (x,v),
\end{equation}
where $\mathcal{D}$ is the class of probability measures in $\mathbb{T}^n \times \mathbb{R}^n$ 
that are invariant under the Euler-Lagrange flow.
In Hamiltonian coordinates the property of invariance for a measure $\nu$ 
can be written more conveniently as:
$$
\int_{\mathbb{T}^n \times \mathbb{R}^n}
\{ \phi, H \} \, d \mu (x,p) = 0, \hspace{1cm} \text{for every } 
 \phi \in C^1_c (\mathbb{T}^n \times \mathbb{R}^n),
$$
where $\mu = \Phi_{\#} \nu$ is the push-forward of the measure $\nu$ 
with respect to the map $\Phi$, i.e., the measure $\mu$ such that 
$$
\int_{\mathbb{T}^n \times \mathbb{R}^n} \phi (x,p) \, d \mu (x,p)
= \int_{\mathbb{T}^n \times \mathbb{R}^n} \phi (x, - D_v L (x,v)) \, d \nu (x,v),
$$
for every $\phi \in C_c (\mathbb{T}^n \times \mathbb{R}^n)$.
Here the symbol $\{ \cdot, \cdot \}$ stands for the Poisson bracket, that is
$$
\{ F, G \} : = D_p F \cdot D_x G - D_x F \cdot D_p G, \hspace{1cm}\text{for every } 
F, G \in C^1 (\mathbb{T}^n \times \mathbb{R}^n).
$$
Denoting by $\mathcal{P} (\mathbb{T}^n \times \mathbb{R}^n)$ 
the class of probability measures on $\mathbb{T}^n \times \mathbb{R}^n$, 
we have
\begin{equation} \label{D}
\mathcal{D} = \left\{ \nu \in \mathcal{P} (\mathbb{T}^n \times \mathbb{R}^n) : \, 
\int_{\mathbb{T}^n \times \mathbb{R}^n}
\{ \phi, H \} \, d \Phi_{\#} \nu (x,p) = 0, \quad \text{for every } 
 \phi \in C^1_c (\mathbb{T}^n \times \mathbb{R}^n) \right\}.
\end{equation}

\medskip

The main disadvantage of problem \eqref{Mather91} is that the set \eqref{D}
where the minimization takes place depends on the Hamiltonian $H$
and thus, in turn, on the integrand $L$.
For this reason, Ricardo Ma\~ne (see \cite{Man}) considered the problem
\begin{equation} \label{MathProb}
\min_{ \nu \in \mathcal{F} }
\int_{\mathbb{T}^n \times \mathbb{R}^n} L ( x , v ) \, d \nu (x,v),
\end{equation}
where
$$
\mathcal{F} : =
\left\{ \nu \in \mathcal{P} ( \mathbb{T}^n \times \mathbb{R}^n ) : \, \int_{\mathbb{T}^n \times \mathbb{R}^n} v \cdot D \psi ( x ) \, d \nu (x,v) = 0,
\quad \text{ for every }\psi \in C^1 ( \mathbb{T}^n ) \right\}.
$$
Measures belonging to $\mathcal{F}$ are called \textit{holonomic} measures.
Notice that, in particular, to every trajectory $\mathbf{y} ( \cdot )$
of the original problem \eqref{min}
we can associate a measure $\nu_{\mathbf{y} (\cdot)} \in \mathcal{F}$.
Indeed, for every $T > 0$ we can first define a measure 
$\nu_{T, \mathbf{y}(\cdot)} \in \mathcal{P} (\mathbb{T}^n \times \mathbb{R}^n)$ by the relation
$$
\int_{\mathbb{T}^n \times \mathbb{R}^n} \phi ( x , v ) \, d \nu_{T, \mathbf{y}(\cdot)} (x,v)
:=  \frac1T \int_0^T \phi (\mathbf{y} (t) , \dot{\mathbf{y}} (t)) \, dt
\quad \text{ for every }\phi \in C_c ( \mathbb{T}^n \times \mathbb{R}^n ).
$$
Then, from the fact that
$$
\text{supp} \, \nu_{T, \mathbf{y}(\cdot)} \subset  \mathbb{T}^n \times [ - M , M ],
\quad \text{ for every }T > 0, 
\quad \quad \text{ ($M=$ Lipschitz constant of $\mathbf{y}(\cdot)$) }
$$
we infer that there exists a sequence $T_j \to \infty$ and a measure 
$\nu_{\mathbf{y}(\cdot)} \in \mathcal{P} (\mathbb{T}^n \times \mathbb{R}^n)$ such that 
$\nu_{T_j, \mathbf{y}(\cdot)} \stackrel{*}{\rightharpoonup} \nu_{\mathbf{y}(\cdot)}$
in the sense of measures, that is, 
\begin{equation} \label{defmu}
\lim_{j \to \infty} \frac{1}{T_j} \int_0^{T_j} \phi (\mathbf{y} (t) , \dot{\mathbf{y}} (t)) \, dt
= \int_{\mathbb{T}^n \times \mathbb{R}^n} \phi ( x , v ) \, d \nu_{\mathbf{y}(\cdot)} (x,v)
\quad \text{ for every } \phi \in C_c ( \mathbb{T}^n \times \mathbb{R}^n).
\end{equation}
Choosing $\phi (x,v) = v \cdot D \psi (x)$ in \eqref{defmu} it follows that $\nu_{\mathbf{y}(\cdot)} \in \mathcal{F}$, since
$$
\int_{\mathbb{T}^n \times \mathbb{R}^n} v \cdot D \psi ( x ) \, d \nu_{\mathbf{y}(\cdot)} (x,v)
= \lim_{j \to \infty} \frac{1}{T_j} \int_0^{T_j}  \dot{\textbf{y}} (t) \cdot D \psi ( \textbf{y} (t) ) \, dt 
= \lim_{j \to \infty}  \frac{ \psi (\textbf{y} (T_j)) - \psi (\textbf{y} (0)) }{T_j} = 0.
$$

\medskip

In principle, since $\mathcal{F}$ is much larger than the class of measures $\mathcal{D}$, 
we could expect the last problem not to have the same solution of \eqref{Mather91}.
However, Ma\~ne proved that every solution of \eqref{MathProb} 
is also a minimizer of \eqref{Mather91}.

A more general version of  \eqref{MathProb} consists in studying,
for each $P \in \mathbb{R}^n$ fixed,
\begin{equation} \label{MathProb2}
\min_{ \nu \in \mathcal{F} }
\int_{\mathbb{T}^n \times \mathbb{R}^n} \left( L ( x , v ) + P \cdot v \right) \, d \nu (x,v),
\end{equation}
referred to as \textit{Mather problem}.
Any minimizer of \eqref{MathProb2} is said to be a \textit{Mather measure}.
An interesting connection between the Mather problem and the time independent Hamilton-Jacobi equation
\eqref{statHJ} 
is established by the identity:
\begin{equation} \label{iden}
- \overline{H} (P) = \min_{ \nu \in \mathcal{F} } 
\int_{\mathbb{T}^n \times \mathbb{R}^n} \left( L ( x , v ) + P \cdot v \right) \, d \nu (x,v).
\end{equation}
Notice that problems \eqref{MathProb} and \eqref{MathProb2}
have the same Euler-Lagrange equation, but possibly different minimizers, 
since the term $P \cdot v$ is a null Lagrangian.
The following theorem gives a characterization of Mather measures in the convex case.
\begin{thm} \label{char}
Let $H: \mathbb{T}^n \times \mathbb{R}^n \to \mathbb{R}$ be a smooth function satisfying (i) and (ii),
and let $P \in \mathbb{R}^n$.
Then, $\nu \in \mathcal{P} (\mathbb{T}^n \times \mathbb{R}^n)$ 
is a solution of \eqref{MathProb2} if and only if:
\begin{align*}
(a)\hspace{.5cm}&\int_{\mathbb{T}^n \times \mathbb{R}^n}
H (x,p) \, d \mu (x,p) = \overline{H} (P) = H (x,p) \quad \quad \mu\text{-a.e.};\\
(b)\hspace{.5cm}&\int_{\mathbb{T}^n \times \mathbb{R}^n}
(p-P) \cdot D_p H (x,p) \, d \mu (x,p) = 0 ; \\
(c)\hspace{.5cm}&\int_{\mathbb{T}^n \times \mathbb{R}^n}
D_p H (x,p) \cdot D \phi (x) \, d \mu (x,p) = 0, \hspace{1cm} \text{for every } 
 \phi \in C^1 (\mathbb{T}^n),
\end{align*}
where $\mu = \Phi_{\#} \nu$ and $\overline{H} (P)$ is defined by Theorem \ref{exist}.
\end{thm}
Before proving Theorem \ref{char} we state the following proposition, which is a consequence
of the results in \cite{Man}, \cite{FATH1}, \cite{FATH2}, \cite{FATH3}, \cite{FATH4} and \cite{EG}.
\begin{prop} \label{oiu}
Let $H: \mathbb{T}^n \times \mathbb{R}^n \to \mathbb{R}$ be a smooth function satisfying (i) and (ii).
 Let $P \in \mathbb{R}^n$, let $\nu \in \mathcal{P} (\mathbb{T}^n \times \mathbb{R}^n)$ 
be a minimizer of \eqref{MathProb2} and set $\mu = \Phi_{\#} \nu$.
Then,
\begin{itemize}

\item[(1)] $\mu$ is invariant under the Hamiltonian dynamics, i.e.
$$
\int_{\mathbb{T}^n \times \mathbb{R}^n}
\{ \phi, H \} \, d \mu (x,p) = 0 \hspace{1cm} \text{for every } 
 \phi \in C^{1}_c (\mathbb{T}^n \times \mathbb{R}^n);
$$

\item[(2)] $\mu$ is supported on the graph 
$$
\Sigma := \{ (x,p) \in \mathbb{T}^n \times \mathbb{R}^n : p = P+D_x u (x,P) \},
$$

\end{itemize}
where $u$ is any viscosity solution of \eqref{statHJ}.
\end{prop}
We observe that property (2), also known as the \textit{graph theorem}, is a highly nontrivial result.
Indeed, by using hypothesis (ii) one can show that any solution $u (\cdot,P)$ of \eqref{statHJ} is Lipschitz continuous,
but higher regularity cannot be expected in the general case. 
\begin{proof}[Proof of Theorem \ref{char}]
To simplify, we will assume $P=0$.

Let $\nu$ be a minimizer of \eqref{MathProb2}.
By the previous proposition, we know that properties (1) and (2) hold;
let us prove that $\mu = \Phi_{\#} \nu$ satisfies $(a)$--$(c)$.
By \eqref{iden}, we have
\[
\int_{\mathbb{T}^n \times \mathbb{R}^n} L (x,v) \, d \nu (x,v) =-\overline{H} (0). 
\]
Furthermore, because of (2)
\[
\int_{\mathbb{T}^n \times \mathbb{R}^n} H (x,p) \, d \mu (x,p)= \overline{H} (0), 
\]
that is, (a). Since $H (x,p) = - L (x,- D_p H (x,p)) + p \cdot D_pH (x,p)$, this implies that 
\[
\int_{\mathbb{T}^n \times \mathbb{R}^n} p \cdot D_pH (x, p) \, d\mu (x,p)=0, 
\]
and so (b) holds. Finally, (c) follows directly from the fact that $\nu \in \mathcal{F}$.

Let now $\mu \in \mathcal{P} ( \mathbb{T}^n \times \mathbb{R}^n )$ satisfy $(a)$--$(c)$, and let us show 
that $\nu= ( \Phi^{-1} )_{\#}  \mu$ is a minimizer of \eqref{MathProb2}. 
First of all, observe that $\nu \in \mathcal{F}$. Indeed, by using (c) for every $\psi \in C^1 ( \mathbb{T}^n )$
$$
\int_{\mathbb{T}^n \times \mathbb{R}^n} v \cdot D \psi ( x ) \, d \nu (x,v) 
= - \int_{\mathbb{T}^n \times \mathbb{R}^n} D_p H (x,p) \cdot D \psi ( x ) \, d \mu (x,p) = 0.
$$
Let now prove that $\nu$ is a minimizer.

Integrating equality $H (x,p) = - L (x, - D_p H (x,p)) + p \cdot D_pH (x,p)$ with respect to $\mu$, 
and using (a) and (b) we have 
\begin{align*}
\overline{H} (0)& = \int_{\mathbb{T}^n \times \mathbb{R}^n} H (x,p) \, d \mu (x,p) \\
&= - \int_{\mathbb{T}^n \times \mathbb{R}^n} L (x, - D_p H (x,p)) \, d \mu (x,p) 
+ \int_{\mathbb{T}^n \times \mathbb{R}^n}  p \cdot D_pH (x,p) \, d \mu (x,p) \\
&= - \int_{\mathbb{T}^n \times \mathbb{R}^n} L (x, - D_p H (x,p)) \, d \mu (x,p)
= - \int_{\mathbb{T}^n \times \mathbb{R}^n} L (x,v) \, d \nu (x,v).
\end{align*}
By \eqref{iden}, $\nu$ is a minimizer of \eqref{MathProb2}.

\end{proof}

\subsection{The Non Convex Case}

The main goal of this paper is to use the techniques of \cite{Ev} and \cite{Tr}
to construct Mather measures under fairly general hypotheses, 
when the variational approach just described cannot be used.
Indeed, when (i) and (ii) are satisfied $H$ coincides with the Legendre transform of $L$, that is, identity $H = H^{**}$ holds.
Moreover, $L$ turns out to be convex and superlinear as well, 
and relation \eqref{Phi} defines a smooth diffeomorphism, 
that allows to pass from Hamiltonian to Lagrangian coordinates.

\medskip

First of all, we extend the definition of Mather measure
to the non convex setting, without making use of the Lagrangian formulation.
\begin{defn} 
We say that a measure $\mu \in \mathcal{P} (\mathbb{T}^n \times \mathbb{R}^n)$ is a \textit{Mather measure} if 
there exists $P \in \mathbb{R}^n$ such that properties (a)--(c) are satisfied.
\end{defn}
The results exposed in the previous subsection show that, modulo the push-forward operation, 
this definition is equivalent to the usual one in literature (see e.g. \cite{F}, \cite{Man}, \cite{Mat}).
We would like now to answer the following natural questions:

\smallskip

\begin{itemize}

\item \textbf{Question 1:} Does a Mather measure exist?

\item \textbf{Question 2:} Let $\mu$ be a Mather measure. Are properties (1) and (2) satisfied?

\end{itemize}
We just showed that in the convex setting 
both questions have affirmative answers.
Before addressing these issues, let us make some hypotheses on the Hamiltonian $H$.
We remark that without any coercivity assumption (i.e. without any condition similar to (ii)),
there are no a priori bounds for the modulus of continuity
of periodic solutions of \eqref{statHJ}.
Indeed, for $n=2$ consider the Hamiltonian 
$$
H(x,p) = p_1^2 - p_2^2 \hspace{2cm} \text{for every } 
p=(p_1,p_2) \in \mathbb{R}^2.
$$
In this case, equation \eqref{statHJ} for $P=0$ and $\overline{H} (P) = 0$ becomes
\begin{equation} \label{counterex}
u_x^2 - u_y^2 = 0.
\end{equation}
Then, for every choice of $f:\mathbb{R} \to \mathbb{R}$ of class $C^1$,
the function $u(x,y) = f(x-y)$ is a solution of \eqref{counterex}.
Clearly, there are no uniform Lipschitz bounds for 
the family of all such functions $u$.
Throughout all the paper, we will assume that 
\begin{itemize}

\item[(H1)] $H$ is smooth;

\item[(H2)] $H (\cdot,p)$ is $\mathbb{Z}^n$-periodic for every $p \in \mathbb{R}^n$;  

\item[(H3)] $ \lim_{|p| \to + \infty}  \left( \dfrac{1}{2}|H(x,p)|^2 + D_x H(x,p) \cdot p \right)= + \infty$
uniformly in $x$.

\end{itemize}
Note that if hypothesis (ii) of the previous subsection holds uniformly in $x$ 
and  we have a bound on $D_x H(x,p)$, e.g. $|D_x H(x,p)| \leq C(1+|p|)$, then
(H3) holds.

First we consider, for every $\varepsilon > 0$, a regularized version of \eqref{statHJ}, 
showing existence and uniqueness of a constant $\overline{H}^{\varepsilon} (P)$ such that 
\begin{equation} \label{statHJepsintro}
- \frac{\varepsilon^2}{2} \Delta u^{\varepsilon} (x) + H (x, P + D u^{\varepsilon} (x) ) = \overline{H}^{\varepsilon} (P)
\end{equation}
admits a $\mathbb{Z}^n$-periodic viscosity (in fact smooth) solution  (see Theorem \ref{appr}).

Thanks to (H3), we can establish a uniform bound on $\| Du^\varepsilon \|_{L^{\infty}}$
and prove that, up to subsequences, $\overline{H}^{\varepsilon} (P) \to \overline{H} (P)$
and  $u^{\varepsilon} (\cdot,P)$ converges uniformly to $u (\cdot,P)$ as $\varepsilon \to 0$,
where $\overline{H} (P)$ and $u (\cdot,P)$ solve equation \eqref{statHJ}.

Observe that, in particular, this shows that Theorem \ref{exist} still holds true
under assumption (H3) when \eqref{coe} does not hold, as for instance when $n=1$ and
\begin{equation*}
H (x,p) = p^3 + V (x), \quad V \text{ smooth and } 
\mathbb{Z}^n \text{-periodic}.
\end{equation*}
On the other hand \eqref{coe} does not imply (H3), see the Hamiltonian
\begin{equation*}
H (x,p) = p^2 \left( 3 + \sin ( e^{p^2} (\cos 2 \pi x ) ) \right) 
\end{equation*}
(here again $n=1$).
Thus, although (H3) seems to be a technical assumption strictly related 
to the particular choice of the approximating equations
\eqref{statHJepsintro}, it is not less general than \eqref{coe}, 
as just clarified by the previous examples.
Anyway, it is not clear at the moment if the results we prove in the present paper 
are still true for Hamiltonians satisfying 
\eqref{coe} but not (H3).

Once suitable properties for the sequence $\{ u^{\varepsilon} \}$ are proved,
for every $\varepsilon > 0$ we define the perturbed Hamilton SDE (see Section \ref{Sec3}) as 
\begin{equation} \label{StocHamDynintr}
\begin{cases}
d \mathbf{x}^{\varepsilon} = - D_p H ( \mathbf{x}^{\varepsilon}, \mathbf{p}^{\varepsilon}) \, dt + \varepsilon \, d w_t, \\
d \mathbf{p}^{\varepsilon} = D_x H (\mathbf{x}^{\varepsilon}, \mathbf{p}^{\varepsilon})\, dt + \varepsilon D^2 u^{\varepsilon} \, d w_t,
\end{cases}
\end{equation}
where $w_t$ is a $n$-dimensional Brownian motion.
The main reason why we use a stochastic approach,
is that in this way we emphasize the connection with the convex setting
by averaging functions along trajectories. 
Nevertheless, our techniques can also be introduced in a purely PDE way
(see Section \ref{SecPDE} 
for a sketch of this approach).

\smallskip
\noindent
In the second step, as just explained, in analogy to what is done in the convex setting 
we encode the long-time behavior of the solutions
$t \mapsto ( \mathbf{x}^{\varepsilon} (t) , \mathbf{p}^{\varepsilon} (t))$ of \eqref{StocHamDynintr}
into a family of probability measures $\{ \mu^{\varepsilon} \}_{\varepsilon > 0}$, defined by
$$
\displaystyle \int_{\mathbb{T}^n \times \mathbb{R}^n} \phi (x,p) \, d \mu^{\varepsilon} (x,p)
:= \lim_{T_j \to \infty} \frac{1}{T_j} \, E \left[ \int_0^{T_j} \phi (\mathbf{x}^{\varepsilon}(t),\mathbf{p}^{\varepsilon}(t)) \, dt \right]
\quad \text{ for every }\phi \in C_c ( \mathbb{T}^n \times \mathbb{R}^n ),
$$
where with $E [\cdot]$ we denote the expected value and
the limit is taken along appropriate subsequences $\{ T_j \}_{j \in \mathbb{N}}$
(see Section \ref{PhSp}).

Using the techniques developed in \cite{Ev}, we are able to provide some bounds
on the derivatives of the functions $u^{\varepsilon}$.
More precisely, defining $\theta_{\mu^{\varepsilon}}$ as the projection on the torus $\mathbb{T}^n$
of the measure $\mu^{\varepsilon}$ (see Section \ref{Proj}), 
we give estimates on the $(L^2, d \theta_{\mu^{\varepsilon}} )$-norm 
of the second and third derivatives of $u^{\varepsilon}$, uniformly w.r.t. $\varepsilon$ (see Proposition \ref{propest}).

In this way, we show that there exist a Mather measure $\mu$ 
and a nonnegative, symmetric $n \times n$ matrix of Borel measures $(m_{kj})_{k,j=1,\ldots,n}$ 
such that $\mu^{\varepsilon}$ converges
weakly to $\mu$ up to subsequences and
\begin{equation} \label{veryimpintro}
\int_{\mathbb{T}^n \times \mathbb{R}^n} \{  \phi , H \}  \, d \mu
+ \int_{\mathbb{T}^n \times \mathbb{R}^n}  
\phi_{p_k p_j} \, d m_{kj} = 0, \hspace{1cm}\forall \, \phi \in C^{2}_c (\mathbb{T}^n \times \mathbb{R}^n),
\end{equation} 
with sum understood over repeated indices (see Theorem \ref{exist2}). 
As in \cite{Ev}, we call $m_{kj}$ the \textit{dissipation measures}.
Relation \eqref{veryimpintro} is the key point of our work, since it immediately shows 
the differences with the convex case.
Indeed, the Mather measure $\mu$ is invariant under the Hamiltonian flow 
if and only the dissipation measures $m_{kj}$ vanish.
When $H(x,\cdot)$ is convex, this is guaranteed by an improved version of 
the estimates on the second derivatives of $u^{\varepsilon}$ (see Proposition \ref{propest}, estimate \eqref{unifcx}).
We give in Section \ref{counter} a one dimensional example showing that,
in general, the dissipation measures $(m_{kj})_{k,j=1,\ldots,n}$ do not disappear.

We study property (2) in Section \ref{comcom}.
In particular, we show that if \eqref{statHJ} admits a solution $u(\cdot,P)$
of class $C^1$, 
which is a rather restrictive condition,
then the corresponding Mather measure $\mu$ given 
by Theorem \ref{exist} 
satisfies
$$
D_p H (x, P + D_x u (x,P))\cdot (p-P-D_x u (x,P))=0
$$
in the support of $\mu$ (see Corollary \ref{rem_comp}).
Observe that this single relation 
is not enough to give us (2) in general, e.g. $n \ge 2$.

Finally, we are able to provide some examples
of non-convex Hamiltonians  (see Section \ref{Ex}), for which both properties (1) and (2) are satisfied.
We observe that the case of strictly quasiconvex Hamiltonians,
which appears among our examples, could also be studied using duality (see Section~\ref{qcx}). 

\end{section}

\begin{section}{Elliptic regularization of the cell problem} \label{sec2}

We start by quoting a classical result concerning an elliptic regularization of equation \eqref{statHJ}.
This, also called vanishing viscosity method,
is a well known tool to study viscosity solutions. In the context 
of Mather measures this procedure was introduced by Gomes in \cite{G}, see also \cite{Nal},
\cite{RenatoNaliniEtc}, \cite{RenatoHector}. 
\begin{thm} \label{appr}
For every $\varepsilon >0$ and every $P \in \mathbb{R}^n$, 
there exists a unique number $\overline{H}^{\varepsilon} (P) \in \mathbb{R}$
such that the equation
\begin{equation} \label{statHJeps}
- \frac{\varepsilon^2}{2} \Delta u^{\varepsilon} (x) + H (x, P + D u^{\varepsilon} (x) ) = \overline{H}^{\varepsilon} (P)
\end{equation}
admits a unique (up to constants) $\mathbb{Z}^n$-periodic viscosity solution.
Moreover, for every $P \in \mathbb{R}^n$
$$
\lim_{\varepsilon \to 0^+} \overline{H}^{\varepsilon} (P)= \overline{H} (P),
\qquad \text{ and } \qquad u^{\varepsilon} \to u \text{ uniformly} \qquad \text{ (up to subsequences), }
$$
where $\overline{H} (P) \in \mathbb{R}$ and $u: \mathbb{T}^n \to \mathbb{R}$ 
are such that \eqref{statHJ} is satisfied in the viscosity sense.
\end{thm}
We call \eqref{statHJeps} the \textit{stochastic cell problem}.
\begin{defn} \label{defL}
Let $\varepsilon >0$ and $P \in \mathbb{R}^n$.
The \textit{linearized operator} $L^{\varepsilon, P}: C^2(\mathbb{T}^n) \to C (\mathbb{T}^n)$ associated to equation \eqref{statHJeps}
is defined as
$$
L^{\varepsilon, P} v (x) :=  - \frac{\varepsilon^2}{2} \Delta v (x)
+ D_p H (x, P + D u^{\varepsilon} (x) ) \cdot D v (x),
$$
for every $v \in C^2(\mathbb{T}^n)$.
\end{defn}

\begin{proof}[Sketch of the Proof]
We mimic the proof in \cite{LPV}.
For every $\lambda > 0$, let's consider the following problem
$$
\lambda v^\lambda + H(x,P+Dv^\lambda) = \dfrac{\varepsilon^2}{2} \Delta v^\lambda.
$$
The above equation has a unique smooth solution $v^\lambda$ in $\mathbb{R}^n$ which is $\mathbb Z^n$-periodic.\\
We will prove that $\| \lambda v^\lambda \|_{L^\infty}, \| Dv^\lambda \|_{L^\infty} \le C$, 
for some positive constant $C$ independent on $\lambda$ and $\varepsilon$.
By using the viscosity property with $\varphi = 0$ as a test function,
we get $\|\lambda v^\lambda\|_{L^\infty} \le C$.
Let now $w^\lambda =\dfrac{ |Dv^\lambda|^2}{2}$.
Then we have
$$
2 \lambda w^\lambda + D_p H \cdot D w^\lambda + D_x H \cdot D v^\lambda 
= \dfrac{\varepsilon^2}{2} \Delta w^\lambda
- \dfrac{\varepsilon^2}{2} |D^2 v^\lambda|^2.
$$
Notice that  for  
$\varepsilon<1/\sqrt{n}$
$$
 \dfrac{\varepsilon^2}{2} |D^2 v^\lambda|^2 \ge  
 \dfrac{\varepsilon^4}{4} |\Delta v^\lambda|^2=(\lambda
v^\lambda +H)^2 \ge \dfrac{1}{2} H^2 - C.
$$
Therefore,
$$
2 \lambda w^\lambda + D_p H \cdot D w^\lambda + D_x H \cdot D v^\lambda 
+ \dfrac{1}{2}H^2 -C \le
\dfrac{\varepsilon^2}{2} \Delta w^\lambda.
$$
At $x_1 \in \mathbb T^n$ where $w^\lambda(x_1) = \max_{\mathbb T^n} w^\lambda$
$$
2 \lambda w^\lambda(x_1) + D_xH \cdot D v^\lambda(x_1) + \dfrac{1}{2} H^2 \le C.
$$
Since $w^\lambda(x_1) \ge 0$, using condition (H3) we deduce that 
$w^\lambda$ is bounded independently of $\lambda, \varepsilon$.
Finally, considering the limit $\lambda \to 0$ we conclude the proof.

\end{proof}

\begin{rem} \label{unifLip}
Bernstein method and (H3) were used in the proof to deduce the uniform bound
on $\|Dv^\lambda\|_{L^\infty}$, which is one of the key properties we need along 
our derivation.
See \cite[Appendix 1]{L} for conditions similar to (H3).

The classical theory (see \cite{L}) 
ensures that the functions $u^{\varepsilon} (\cdot, P)$ are $C^{\infty}$.
In addition, the previous proof shows that they are Lipschitz, 
with Lipschitz constant independent of $\varepsilon$.
\end{rem}

\end{section}

\begin{section}{Stochastic dynamics} \label{Sec3}

We now introduce a stochastic dynamics associated with the stochastic cell problem \eqref{statHJeps}.
This will be a perturbation to the Hamiltonian dynamics \eqref{HE},
which describes the trajectory in the phase space 
of a classical mechanical system. 

Let 
$(\mathbb{T}^n, \sigma, P)$
be a probability space, and let $w_t$ be a $n$-dimensional Brownian motion on 
$\mathbb{T}^n$.
Let $\varepsilon >0$, and let $u^{\varepsilon}$
be a $\mathbb{Z}^n$-periodic solution of \eqref{statHJeps}.
To simplify, we set $P = 0$. 
Consider now the solution $\mathbf{x}^\varepsilon(t)$ of 
\begin{equation} \label{stoc}
\begin{cases} 
d \mathbf{x}^{\varepsilon} = - D_p H ( \mathbf{x}^{\varepsilon} , 
D u^{\varepsilon} ( \mathbf{x}^{\varepsilon} ) ) \, dt + \varepsilon \, d w_t, \\
\mathbf{x}^{\varepsilon}(0) = \overline{x},
\end{cases}
\end{equation}
with $\overline{x} \in \mathbb{T}^n$ arbitrary.
Accordingly, the momentum variable is defined as
$$
\mathbf{p}^{\varepsilon} (t) = D u^{\varepsilon} (\mathbf{x}^{\varepsilon}(t)).
$$
\begin{rem} \label{pbdd}
From Remark \ref{unifLip} it follows that
$$
\sup_{t > 0} | \mathbf{p}^{\varepsilon} (t) | < \infty.
$$ 
\end{rem}
Let us now recall some basic fact of stochastic calculus.
Suppose $\mathbf{z} : [0, +\infty) \to \mathbb{R}^n$ is a solution to the SDE: 
$$
d \mathbf{z}_i = a_i \, dt + b_{ij} \, w^{j}_t \hspace{2cm}i=1,\ldots,n,
$$
with $a_{i}$ and $b_{ij}$ bounded and progressively measurable processes.
Let $\varphi: \mathbb{R}^n \times \mathbb{R} \to \mathbb{R}$
be a smooth function.
Then, $\varphi (\mathbf{z},t)$ satisfies the \textit{It\^o formula}:
\begin{equation} \label{Ito}
d \varphi = \varphi_{z_i} \, d z_i + \left( \varphi_t 
+ \frac{1}{2} b_{ij} b_{jk} \varphi_{z_i z_k} \right) dt.
\end{equation}
An integrated version of the It\^o formula is the \textit{Dynkin's formula}:
$$
E \left[ \phi (\mathbf{z}(T)) - \phi (\mathbf{z}(0)) \right] 
= E \left[ \int_0^T \left( a_i D_{z_i} \phi (\mathbf{z} (t)) + \frac{1}{2} b_{ij} b_{jk} D^2_{z_i z_k} \phi (\mathbf{z}(t)) \right) \, dt \right].
$$
Here and always in the sequel, we use Einstein's convention for repeated indices in a sum.
In the present situation, we have
$$
a_i = - D_{p_i} H ( \mathbf{x}^{\varepsilon} , D u^{\varepsilon} ) , \hspace*{1cm}
b_{ij} = \varepsilon \delta_{ij}.
$$
Hence, recalling \eqref{stoc} and \eqref{Ito}
\begin{align*}
d p_i &= u^{\varepsilon}_{x_i x_j} \, d x^{\varepsilon}_j 
+ \frac{\varepsilon^2}{2} \Delta ( u^{\varepsilon}_{x_i} ) \, dt 
= - L^{\varepsilon, P} u^{\varepsilon}_{x_i} dt 
+ \varepsilon u^{\varepsilon}_{x_i x_j} \, d w^j_t \\
&=  D_{x_i} H \, dt + \varepsilon u^{\varepsilon}_{x_i x_j} \, d w^j_t, 
\end{align*}
where in the last equality we used identity \eqref{dbeta}.
Thus, $(\mathbf{x}^{\varepsilon},\mathbf{p}^{\varepsilon})$ satisfies the following stochastic version
of the Hamiltonian dynamics \eqref{HE}:
\begin{equation} \label{StocHamDyn}
\begin{cases}
d \mathbf{x}^{\varepsilon} = - D_p H ( \mathbf{x}^{\varepsilon}, \mathbf{p}^{\varepsilon}) \, dt 
+ \varepsilon \, d w_t, \\
d \mathbf{p}^{\varepsilon} = D_x H (\mathbf{x}^{\varepsilon}, \mathbf{p}^{\varepsilon})\, dt 
+ \varepsilon D^2 u^{\varepsilon} \, d w_t.
\end{cases}
\end{equation}
We are now going to study the behavior of the solutions $u^{\varepsilon}$
of equation \eqref{statHJeps} along the trajectory $\mathbf{x}^{\varepsilon} (t)$.
Thanks to the It\^o formula and relations \eqref{StocHamDyn}
and \eqref{statHJeps}:
\begin{align}
d u^{\varepsilon} (\mathbf{x}^{\varepsilon}(t))
&= D u^{\varepsilon} d \mathbf{x}^{\varepsilon}  + \frac{\varepsilon^2}{2} \Delta u^{\varepsilon} \, dt
=  - L^{\varepsilon, P} u^{\varepsilon} dt
+ \varepsilon D u^{\varepsilon} \, d w_t \nonumber \\
&= \left( H - \overline{H}^{\varepsilon} - D u^{\varepsilon}\cdot D_p H \right) dt
+ \varepsilon D u^{\varepsilon} \, d w_t. \label{eqnu}
\end{align}
Using Dynkin's formula in \eqref{eqnu} we obtain
$$
E \big( u^{\varepsilon} (\mathbf{x}^{\varepsilon}(T)) - u^{\varepsilon} (\mathbf{x}^{\varepsilon}(0)) \big)
= E \left[ \int_0^T \left( H - \overline{H}^{\varepsilon} - D u^{\varepsilon} \cdot D_p H \right) dt \right].
$$
We observe that in the convex case, since the Lagrangian $L$ 
is related with the Hamiltonian by the relation
$$
L = p \cdot D_p H - H, 
$$
we have
$$
u^{\varepsilon} (\mathbf{x}^{\varepsilon}(0) ) 
=  E \left[ \int_0^T ( L + \overline{H}^{\varepsilon} ) \, dt 
+ u^{\varepsilon} (\mathbf{x}^{\varepsilon}(T)) \right].
$$

\begin{subsection}{Phase space measures} \label{PhSp}

We will encode the asymptotic behaviour of the trajectories 
by considering ergodic averages.
More precisely, we associate to every trajectory $( \mathbf{x}^{\varepsilon}(\cdot) , \mathbf{p}^{\varepsilon}(\cdot) )$
of \eqref{StocHamDyn} a probability measure $\mu^{\varepsilon} \in \mathcal{P} (\mathbb{T}^n \times \mathbb{R}^n)$ 
defined by
\begin{equation} \label{limT}
\displaystyle \int_{\mathbb{T}^n \times \mathbb{R}^n} \phi (x,p) \, d \mu^{\varepsilon} (x,p)
:= \lim_{T \to \infty} \frac{1}{T} \, E \left[ \int_0^T \phi (\mathbf{x}^{\varepsilon}(t),\mathbf{p}^{\varepsilon}(t)) \, dt \right],
\end{equation}
for every $\phi \in C_c ( \mathbb{T}^n \times \mathbb{R}^n )$.
In the expression above, the definition makes sense 
provided the limit is taken over an appropriate subsequence.
Moreover, no uniqueness is asserted, since by choosing 
a different subsequence one can in principle obtain a different limit measure $\mu^{\varepsilon}$.
Then, using Dynkin's formula we have, for every $\phi \in C^{2}_c (\mathbb{T}^n \times \mathbb{R}^n)$,
\begin{align}
&E \left[ \phi ( \mathbf{x}^{\varepsilon}(T), \mathbf{p}^{\varepsilon}(T)) - \phi (\mathbf{x}^{\varepsilon}(0), \mathbf{p}^{\varepsilon}(0)) \right]
= E \left[ \int_0^T \Big( D_p \phi \cdot D_x H - D_x \phi \cdot D_p H \Big) \, dt \right] \nonumber \\
&\hspace{1cm}+ E \left[ \int_0^T \Big(\frac{\varepsilon^2}{2} \phi_{x_i x_i} 
+ \varepsilon^2  u^{\varepsilon}_{x_i x_j} \phi_{x_i p_j} 
+ \frac{\varepsilon^2}{2} u^{\varepsilon}_{x_i x_k} 
u^{\varepsilon}_{x_i x_j}  \phi_{p_k p_j} \Big) \, dt \right]. \label{exp}
\end{align}
Dividing last relation by $T$ and passing to the limit as $T \to + \infty$ 
(along a suitable subsequence) we obtain
\begin{equation} \label{imp}
\int_{\mathbb{T}^n \times \mathbb{R}^n} \{  \phi , H \}  \, d \mu^{\varepsilon}
+ \int_{\mathbb{T}^n \times \mathbb{R}^n}  
\left[ \frac{\varepsilon^2}{2} \phi_{x_i x_i} 
+ \varepsilon^2 u^{\varepsilon}_{x_i x_j} \phi_{x_i p_j} 
+ \frac{\varepsilon^2}{2}  u^{\varepsilon}_{x_i x_k} 
u^{\varepsilon}_{x_i x_j} \phi_{p_k p_j}
\right] \, d \mu^{\varepsilon} = 0.
\end{equation}

\end{subsection}

\begin{subsection}{Projected measure} \label{Proj}

We define the \textit{projected measure} $\theta_{\mu^{\varepsilon}} \in \mathcal{P} (\mathbb{T}^n)$ 
in the following way:
$$
\int_{\mathbb{T}^n} \varphi (x) \, d \theta_{\mu^{\varepsilon}} (x)
:= \int_{\mathbb{T}^n \times \mathbb{R}^n} \varphi (x) \, d \mu^{\varepsilon} (x,p),
\hspace{2cm}\forall \, \varphi \in C (\mathbb{T}^n).
$$
Using test functions that do not depend 
on the variable $p$ in the previous definition
we conclude from identity \eqref{imp} that
\begin{equation} \label{po}
\int_{\mathbb{T}^n}  D_p H \cdot D \varphi \, \, d \theta_{\mu^{\varepsilon}} 
=  \frac{\varepsilon^2}{2} \int_{\mathbb{T}^n \times \mathbb{R}^n} \Delta \varphi 
\, d \theta_{\mu^{\varepsilon}} , \hspace{2cm}\forall \, \varphi \in C^2 (\mathbb{T}^n).
\end{equation}

\end{subsection}

\begin{subsection}{PDE Approach} \label{SecPDE}

The measures $\mu^{\varepsilon}$ and $\theta_{\mu^{\varepsilon}}$
can be defined also by using standard PDE methods from \eqref{po}.
Indeed, given $u^{\varepsilon}$ we can consider the PDE
$$
\frac{\varepsilon^2}{2} \Delta \theta^{\varepsilon}
+ \textnormal{div} \left( D_p H (x, D u^{\varepsilon}) \, \theta^{\varepsilon} \right) = 0,
$$
which admits a unique non-negative solution $\theta^{\varepsilon}$ with
$$
\int_{\mathbb{T}^n} \theta^{\varepsilon} (x) \, dx= 1,
$$
since it is not hard to see that $0$ is the principal eigenvalue of the 
following elliptic operator in $C^2 (\mathbb T^n)$:
$$
v \longmapsto -\frac{\varepsilon^2}{2} \Delta v-\text{div} ( D_p H (x, D
u^{\varepsilon}) \, v).
$$ 
Then $\mu^{\varepsilon}$ can be defined as a unique measure such that
$$
\int_{\mathbb{T}^n \times \mathbb{R}^n} \psi (x,p) \, d \mu^{\varepsilon} (x,p)
= \int_{\mathbb{T}^n} \psi (x, D u^{\varepsilon} (x)) \, d \theta^{\varepsilon} (x),
$$
for every $\psi \in C_c (\mathbb T^n \times \mathbb R^n)$.
Finally, identity \eqref{imp} requires some work but can also be proved in a purely analytic way.
\end{subsection}

\end{section}

\begin{section}{Uniform estimates} \label{Est}

In this section we derive several estimates that will be useful when passing to the limit as $\varepsilon \to 0$.
We will use here the same techniques as in \cite{Ev} and \cite{Tr}.

\begin{prop} \label{propest}
We have the following estimates:
\begin{align}
&\varepsilon^2 \int_{\mathbb{T}^n} | D^2_{xx} u^{\varepsilon} |^2 \, d \theta_{\mu^{\varepsilon}} \leq C,
\label{est} \\
&\varepsilon^2 \int_{\mathbb{T}^n} |D^2_{P x} u^{\varepsilon} |^2 \, d \theta_{\mu^{\varepsilon}} 
\leq \int_{\mathbb{T}^n} | D_P u^{\varepsilon} |^2 \, d \theta_{\mu^{\varepsilon}} 
+ \int_{\mathbb{T}^n} | D_p H - D_P \overline{H}^{\varepsilon} |^2 \, d \theta_{\mu^{\varepsilon}}, \label{est2} \\
& \varepsilon^2 \int_{\mathbb{T}^n} |Du^{\varepsilon}_{x_i x_i}|^2 \, d \theta_{\mu^{\varepsilon}} 
\leq C \left(1 +  \int_{\mathbb{T}^n} |D^2_{xx} u^{\varepsilon} |^3 \, d \theta_{\mu^{\varepsilon}} \right),
 \qquad i=1,\ldots,n. \label{est3}
\end{align}
In addition, if $H$ is uniformly convex in $p$, inequalities \eqref{est} and \eqref{est2} can be improved to:
\begin{align} 
&\int_{\mathbb{T}^n}  |D^2_{xx} u^{\varepsilon}|^2 \, d \theta_{\mu^{\varepsilon}} \leq C, \label{unifcx} \\
&\int_{\mathbb{T}^n} |D^2_{P x} u^{\varepsilon} |^2 \, d \theta_{\mu^{\varepsilon}}
\leq C \, \textnormal{trace} \, ( D^2_{PP} \overline{H}^{\varepsilon} ), \label{unifcx2} 
\end{align}
respectively.
Here $C$ denotes a positive constant independent of $\varepsilon$.
\end{prop}

\begin{rem}
Estimate \eqref{unifcx} was already proven in \cite{Ev} and \cite{Tr}.
\end{rem}

To prove the proposition we first need an auxiliary lemma.
In the following, we denote by $\beta$ either a direction in $\mathbb{R}^n$
(i.e. $\beta \in \mathbb{R}^n$ with $|\beta|=1$), or a parameter 
(e.g. $\beta = P_i$ for some $i \in \{ 1, \ldots, n\}$). 
When $\beta = P_i$ for some $i \in \{ 1, \ldots, n\}$ 
the symbols $H_{\beta}$ and $H_{\beta \beta}$ have to be understood as $H_{p_i}$
and $H_{p_i p_i}$, respectively.

\begin{lem}
We have
\begin{align}
&\varepsilon^2 \int_{\mathbb{T}^n} | D_x u^{\varepsilon}_{\beta} |^2 \, d \theta_{\mu^{\varepsilon}}
= 2 \int_{\mathbb{T}^n} u^{\varepsilon}_{\beta} (  \overline{H}^{\varepsilon}_{\beta} - H_{\beta} ) \, d \theta_{\mu^{\varepsilon}}, \label{estim}\\
&\int_{\mathbb{T}^n} ( \overline{H}^{\varepsilon}_{\beta \beta}
- H_{\beta \beta}  - 2 D_{p} H_{\beta} \cdot D_x u^{\varepsilon}_{\beta}
- D^2_{p p} H D_x u^{\varepsilon}_{\beta} \cdot D_x u^{\varepsilon}_{\beta} ) \, d \theta_{\mu^{\varepsilon}}
= 0, \label{estim12}\\
&\varepsilon^2 \int_{\mathbb{T}^n} | D_x u^{\varepsilon}_{\beta \beta} |^2 \, d \theta_{\mu^{\varepsilon}} 
= 2 \int_{\mathbb{T}^n} u^{\varepsilon}_{\beta \beta} ( \overline{H}^{\varepsilon}_{\beta \beta}
- H_{\beta \beta}  - 2 D_{p} H_{\beta} \cdot D_x u^{\varepsilon}_{\beta}
- D^2_{p p} H : D_x u^{\varepsilon}_{\beta} \otimes D_x u^{\varepsilon}_{\beta} \, d
\theta_{\mu^{\varepsilon}}. \label{estim2}
\end{align}
\end{lem}

\begin{proof}
By differentiating equation \eqref{statHJeps} with respect to $\beta$ and recalling Definition \ref{defL} we get
\begin{equation} \label{dbeta}
L^{\varepsilon, P} u^{\varepsilon}_{\beta} = \overline{H}^{\varepsilon}_{\beta} - H_{\beta},
\end{equation}
so that
\begin{align*}
\frac{1}{2} L^{\varepsilon, P} ( |u^{\varepsilon}_{\beta}|^2)  = u^{\varepsilon}_{\beta} L^{\varepsilon, P} u^{\varepsilon}_{\beta}
- \frac{\varepsilon^2}{2} | D_x u^{\varepsilon}_{\beta} |^2
= u^{\varepsilon}_{\beta} (  \overline{H}^{\varepsilon}_{\beta} - H_{\beta} )
- \frac{\varepsilon^2}{2} | D_x u^{\varepsilon}_{\beta} |^2.
\end{align*}
Integrating w.r.t. $\theta_{\mu^{\varepsilon}}$ and recalling \eqref{po} we get \eqref{estim}.

To prove \eqref{estim12}, we differentiate \eqref{dbeta} w.r.t. $\beta$ obtaining
\begin{equation} \label{d2beta}
L^{\varepsilon, P} u^{\varepsilon}_{\beta \beta} 
= \overline{H}^{\varepsilon}_{\beta \beta}
- H_{\beta \beta}  - 2 D_{p} H_{\beta} \cdot D_x u^{\varepsilon}_{\beta}
- D^2_{p p} H : D_x u^{\varepsilon}_{\beta} \otimes D_x u^{\varepsilon}_{\beta}.
\end{equation}
Integrating w.r.t. $\theta_{\mu^{\varepsilon}}$ and recalling \eqref{po} equality \eqref{estim12} follows.
Finally, using \eqref{d2beta}
\begin{align*}
&\frac{1}{2} L^{\varepsilon, P} ( |u^{\varepsilon}_{\beta \beta}|^2)  
= u^{\varepsilon}_{\beta \beta} L^{\varepsilon, P} u^{\varepsilon}_{\beta \beta}
- \frac{\varepsilon^2}{2} | D_x u^{\varepsilon}_{\beta \beta} |^2 \\
&\hspace{.5cm}= u^{\varepsilon}_{\beta \beta} 
( \overline{H}^{\varepsilon}_{\beta \beta}
- H_{\beta \beta}  - 2 D_{p} H_{\beta} \cdot D_x u^{\varepsilon}_{\beta}
- D^2_{p p} H : D_x u^{\varepsilon}_{\beta} \otimes D_x u^{\varepsilon}_{\beta} )
- \frac{\varepsilon^2}{2} | D_x u^{\varepsilon}_{\beta \beta} |^2.
\end{align*}
Once again, we integrate w.r.t. $\theta_{\mu^{\varepsilon}}$ and use \eqref{po}
to get \eqref{estim2}.
\end{proof}
We can now proceed to the proof of Proposition \ref{propest}.

\begin{proof}[Proof of Proposition \ref{propest}]

Summing up the $n$ identities obtained from \eqref{estim} 
with $\beta = x_1,\ldots, x_n$ respectively, we have
$$
\varepsilon^2 \int_{\mathbb{T}^n} | D^2_{xx} u^{\varepsilon} |^2 \, d \theta_{\mu^{\varepsilon}}
= -2 \int_{\mathbb{T}^n} D_x u^{\varepsilon} \cdot D_x H \, d \theta_{\mu^{\varepsilon}}.
$$
Thanks to Remark \ref{unifLip}, \eqref{est} follows.
Analogously, relation \eqref{est2} is obtained by summing up \eqref{estim} 
with $\beta = P_1, P_2, \ldots, P_n$, which yields
$$
\varepsilon^2  \int_{\mathbb{T}^n} |D^2_{P x} u^{\varepsilon} |^2 \, d \theta_{\mu^{\varepsilon}} 
= 2 \int_{\mathbb{T}^n} D_P u^{\varepsilon} \cdot 
\left[ D_P \overline{H}^{\varepsilon} - D_p H \right] \, d \theta_{\mu^{\varepsilon}}.
$$
Let us show \eqref{est3}.
Thanks to \eqref{estim2}
\begin{align*} 
&\varepsilon^2 \int_{\mathbb{T}^n} | D_x u^{\varepsilon}_{x_i x_i} |^2 \, d \theta_{\mu^{\varepsilon}}
\nonumber \\
&\hspace{.5cm}= - 2 \int_{\mathbb{T}^n} u^{\varepsilon}_{x_i x_i} 
( H_{x_i x_i}  + 2 D_{p} H_{x_i} \cdot D_x u^{\varepsilon}_{x_i}
+ D^2_{p p} H : D_x u^{\varepsilon}_{x_i} \otimes D_x u^{\varepsilon}_{x_i} ) 
\, d \theta_{\mu^{\varepsilon}}. 
\end{align*}
Since the functions $u^{\varepsilon}$ are uniformly Lipschitz, we have
$$
| H_{x_i x_i} |, \, | D_{p} H_{x_i} |, \,  |D^2_{p p} H| \leq C,
\quad \text{ on the support of }\theta_{\mu^{\varepsilon}}.
$$
Hence, 
\begin{align*}
\varepsilon^2 \int_{\mathbb{T}^n} | D_x u^{\varepsilon}_{x_i x_i} |^2 \, d \theta_{\mu^{\varepsilon}}
&\leq C \left[  \int_{\mathbb{T}^n} |D^2_{xx} u^{\varepsilon} | \, d \theta_{\mu^{\varepsilon}} 
+  \int_{\mathbb{T}^n} |D^2_{xx} u^{\varepsilon} |^2 \, d \theta_{\mu^{\varepsilon}} 
+  \int_{\mathbb{T}^n} |D^2_{xx} u^{\varepsilon} |^3 \, d \theta_{\mu^{\varepsilon}}  \right] \\
&\leq C \left(1 +  \int_{\mathbb{T}^n} |D^2_{xx} u^{\varepsilon} |^3 \, d \theta_{\mu^{\varepsilon}} \right).
\end{align*}
Finally, assume that $H$ is uniformly convex.
Thanks to \eqref{estim12} for every $i=1,\ldots,n$
\begin{align*}
0 &= \int_{\mathbb{T}^n}  ( H_{x_i x_i}  + 2 D_{p} H_{x_i} \cdot D_x u^{\varepsilon}_{x_i}
+ D^2_{p p} H D_x u^{\varepsilon}_{x_i} \cdot D_x u^{\varepsilon}_{x_i} ) \, d \theta_{\mu^{\varepsilon}} \\
&\geq \int_{\mathbb{T}^n}  \left(  H_{x_i x_i} + 2 D_{p} H_{x_i} \cdot D_x u^{\varepsilon}_{x_i}  \right) \,
d \theta_{\mu^{\varepsilon}}
+ \alpha \| D_x u^{\varepsilon}_{x_i}  \|^2_{L^2(\mathbb{T}^n;d \theta_{\mu^{\varepsilon}})},
\end{align*}
for some $\alpha > 0$.
Thus, using Cauchy's and Young's inequalities, for every $\eta \in \mathbb{R}$
\begin{align*}
&\alpha \| D_x u^{\varepsilon}_{x_i}  \|^2_{L^2(\mathbb{T}^n;d \theta_{\mu^{\varepsilon}})} 
\leq - \int_{\mathbb{T}^n} H_{x_i x_i} \, d \theta_{\mu^{\varepsilon}}
+ 2 \| D_{p} H_{x_i}  \|_{L^2(\mathbb{T}^n;d \theta_{\mu^{\varepsilon}})} 
\| D_x u^{\varepsilon}_{x_i}  \|_{L^2(\mathbb{T}^n;d \theta_{\mu^{\varepsilon}})} \\
&\hspace{.5cm}\leq - \int_{\mathbb{T}^n} H_{x_i x_i} \, d \theta_{\mu^{\varepsilon}}
+ \frac{1}{\eta^2} \| D_{p} H_{x_i}  \|^2_{L^2(\mathbb{T}^n;d \theta_{\mu^{\varepsilon}})} 
+ \eta^2 \| D_x u^{\varepsilon}_{x_i}  \|^2_{L^2(\mathbb{T}^n;d \theta_{\mu^{\varepsilon}})} .
\end{align*}
Finally, 
$$
( \alpha - \eta^2) \| D_x u^{\varepsilon}_{x_i}  \|^2_{L^2(\mathbb{T}^n;d \theta_{\mu^{\varepsilon}})} 
\leq - \int_{\mathbb{T}^n} H_{x_i x_i} \, d \theta_{\mu^{\varepsilon}}
+ \frac{1}{\eta^2} \| D_{p} H_{x_i}  \|^2_{L^2(\mathbb{T}^n;d \theta_{\mu^{\varepsilon}})} .
$$
Choosing $\eta^2 < \alpha$ we get \eqref{unifcx}.

Let $i \in \{ 1, \ldots, n \}$ and let us integrate w.r.t. $\theta_{\mu^{\varepsilon}}$
relation \eqref{d2beta} with $\beta = P_i$:
$$
0 = \int_{\mathbb{T}^n} ( \overline{H}^{\varepsilon}_{P_i P_i}
- H_{p_i p_i}  - 2 D_{p} H_{p_i} \cdot D_x u^{\varepsilon}_{P_i}
- D^2_{p p} H D_x u^{\varepsilon}_{P_i} \cdot D_x u^{\varepsilon}_{P_i} ) \, d \theta_{\mu^{\varepsilon}}.
$$
Since $D^2_{pp} H$ is positive definite, 
\begin{align*}
& \alpha \int_{\mathbb{T}^n} | D_x u^{\varepsilon}_{P_i} |^2 \, d \theta_{\mu^{\varepsilon}}
\leq \int_{\mathbb{T}^n} ( \overline{H}^{\varepsilon}_{P_i P_i}
- H_{p_i p_i}  - 2 D_{p} H_{p_i} \cdot D_x u^{\varepsilon}_{P_i} ) \, d \theta_{\mu^{\varepsilon}} \\
&\hspace{.5cm} \leq \int_{\mathbb{T}^n} ( \overline{H}^{\varepsilon}_{P_i P_i}
- 2 D_{p} H_{p_i} \cdot D_x u^{\varepsilon}_{P_i} ) \, d \theta_{\mu^{\varepsilon}}.
\end{align*}
Using once again Cauchy's and Young's inequalities and summing up 
with respect to $i = 1,\ldots,n$ \eqref{unifcx2} follows.
\end{proof}

\end{section}

\begin{section}{Existence of Mather measures and dissipation measures}

We now look at the asymptotic behavior of the measures $\mu^{\varepsilon}$ as $\varepsilon \to 0$,
proving existence of Mather measures.
The main result of the section is the following.
\begin{thm} \label{exist2}
Let $H: \mathbb{T}^n \times \mathbb{R}^n \to \mathbb{R}$ 
be a smooth function satisfying conditions (H1)--(H3),
and let $\{ \mu^{\varepsilon} \}_{\varepsilon > 0}$ 
be the family of measures defined in Section \ref{Sec3}.
Then there exist a Mather measure $\mu$ and a nonnegative, 
symmetric $n \times n$ matrix $( m_{kj} )_{k,j=1,\ldots n}$ of Borel measures such that
\begin{equation} \label{weakconv}\\
\mu^{\varepsilon} \stackrel{*}{\rightharpoonup} \mu \quad
\text{ in the sense of measures up to subsequences},
\end{equation}
and
\begin{equation} \label{veryimp}
\int_{\mathbb{T}^n \times \mathbb{R}^n} \{  \phi , H \}  \, d \mu
+ \int_{\mathbb{T}^n \times \mathbb{R}^n}  
\phi_{p_k p_j} \, d m_{kj} = 0, \hspace{1cm}\forall \, \phi \in C^2_c (\mathbb{T}^n \times \mathbb{R}^n).
\end{equation} 
Moreover, 
\begin{equation} \label{suppcompact}
\textnormal{supp} \, \mu \text{ and }\textnormal{supp} \, m \text{ are compact}. 
\end{equation}
We call the matrix $m_{kj}$ the \textit{dissipation measure}.
\end{thm}

\begin{proof}
First of all, we notice that since we have a uniform (in $\varepsilon$)
Lipschitz estimate for the functions $u^{\varepsilon}$, 
there exists a compact set $K \subset \mathbb{T}^n \times \mathbb{R}^n$
such that
$$
\text{supp}\,\mu^{\varepsilon} \subset K, \hspace{2 cm}\forall \, \varepsilon > 0.
$$
Moreover, up to subsequences, we have \eqref{weakconv}, that is
$$
\lim_{\varepsilon \to 0} 
\int_{\mathbb{T}^n \times \mathbb{R}^n} \phi \, d \mu^{\varepsilon}
\to \int_{\mathbb{T}^n \times \mathbb{R}^n} \phi \, d \mu, 
$$
for every function $\phi \in C_c (\mathbb{T}^n \times \mathbb{R}^n)$, 
for some probability measure $\mu \in \mathcal{P} (\mathbb{T}^n \times \mathbb{R}^n)$,
and this proves \eqref{weakconv}.
From what we said, it follows that
$$
\text{supp}\, \mu \subset K.
$$
To show \eqref{veryimp}, we need to pass to the limit in relation \eqref{imp}.
First, let us focus on the second term of the aforementioned formula:
\begin{equation} \label{imp2}
\int_{\mathbb{T}^n \times \mathbb{R}^n}  
\left[ \frac{\varepsilon^2}{2} \phi_{x_i x_i} 
+ \varepsilon^2 u^{\varepsilon}_{x_i x_j} \phi_{x_i p_j} 
+ \frac{\varepsilon^2}{2}  u^{\varepsilon}_{x_i x_k} 
u^{\varepsilon}_{x_i x_j} \phi_{p_k p_j}
\right] \, d \mu^{\varepsilon}.
\end{equation}
By the bounds of the previous section,
$$
\lim_{\varepsilon \to 0} 
\int_{\mathbb{T}^n \times \mathbb{R}^n}  
\left[ \frac{\varepsilon^2}{2} \phi_{x_i x_i} 
+ \varepsilon^2 u^{\varepsilon}_{x_i x_j} \phi_{x_i p_j} 
\right] \, d \mu^{\varepsilon} = 0.
$$
However, as in \cite{Ev}, the last term in \eqref{imp2}
does not vanish in the limit.
In fact, through a subsequence, for every $k,j=1,\ldots,n$ we have
$$
\frac{\varepsilon^2}{2} \int_{\mathbb{T}^n \times \mathbb{R}^n}  
u^{\varepsilon}_{x_i x_k} 
u^{\varepsilon}_{x_i x_j} \psi (x,p) \, d \mu^{\varepsilon} (x,p)
\longrightarrow \int_{\mathbb{T}^n \times \mathbb{R}^n}  \psi (x,p) \, dm_{kj} (x,p)
\quad \forall \, \psi \in C_c (\mathbb{T}^n \times \mathbb{R}^n),
$$
for some nonnegative, symmetric $n \times n$ matrix $( m_{kj} )_{k,j=1,\ldots n}$ of Borel measures.
Passing to the limit as $\varepsilon \to 0$ in \eqref{imp} condition \eqref{veryimp} follows.
From Remark \ref{pbdd} we infer that $\text{supp} \, m \subset K$,
so that \eqref{suppcompact} follows.

Let us show that $\mu$ satisfies conditions (a)--(c) with $P = 0$. 
As in \cite{Ev} and \cite{Tr}, consider 
$$
\int_{\mathbb{T}^n \times \mathbb{R}^n} \left( 
H (x, p ) - \overline{H}^{\varepsilon} \right)^2 \, d \mu^{\varepsilon} (x,p)
= \frac{\varepsilon^4}{4} \int_{\mathbb{T}^n \times \mathbb{R}^n} 
| \Delta u^{\varepsilon} (x) |^2 \, d \mu^{\varepsilon} (x,p)  \longrightarrow 0
$$
as $\varepsilon \to 0$, where we used \eqref{statHJeps} and \eqref{est}.
Therefore, (a) follows.
Let us consider relation \eqref{imp}, and let us choose
as test function $\phi = \varphi (u^{\varepsilon})$. 
We get
$$
\int_{\mathbb{T}^n \times \mathbb{R}^n} \varphi ' (u^{\varepsilon}) D_x u^{\varepsilon}
\cdot D_p H \, d \mu^{\varepsilon} 
+ \varepsilon^2 \int_{\mathbb{T}^n \times \mathbb{R}^n} 
\left( \varphi ' (u^{\varepsilon}) u^{\varepsilon}_{x_i x_i} 
+ \varphi '' (u^{\varepsilon}) (u^{\varepsilon}_{x_i})^2 \right) \, d \mu^{\varepsilon} = 0.
$$
Passing to the limit as $\varepsilon \to 0$, we have
$$
\int_{\mathbb{T}^n \times \mathbb{R}^n} \varphi ' (u ) \, p \cdot D_p H \, d \mu = 0.
$$
Choosing $\varphi (u) = u$ we get (b). 
Finally, relation (c) follows by simply choosing in \eqref{veryimp} 
test functions $\phi$ that do not depend on the variable $p$.

\end{proof}
We conclude the section with a useful identity that will be used in Section \ref{Ex}.

\begin{prop}
 For every $\lambda \in \mathbb{R}$
\begin{equation} \label{iul}
\int_{\mathbb{T}^n \times \mathbb{R}^n}  e^{\lambda H} \left(
\lambda H_{p_k} H_{p_j} + H_{p_k p_j} \right) \, d m_{kj} = 0.
\end{equation}
\end{prop}
\begin{proof}
First recall that for any function $f: \mathbb{R} \to \mathbb{R}$
of class $C^1$
$$
\{ H, f (H) \} = 0,
$$
and, furthermore, for any $\psi \in C^1 (\mathbb{T}^n \times \mathbb{R}^n)$
$$
\{ H, \psi f (H) \} = \{ H, \psi \} \, f(H).
$$
Let now $\lambda \in \mathbb{R}$.
By choosing in \eqref{veryimp} $\phi = \psi f(H)$ 
with $f(z) = e^{\lambda z}$ and $\psi \equiv 1$ we conclude the proof.
\end{proof}

\end{section}

\begin{section}{Support of the dissipation measures} 

We discuss now in a more detailed way the structure of $\text{supp} \, m$.
\begin{prop}
We have
\begin{equation} \label{Suppm}
\text{supp} \, m \subset \overline{
\bigcup_{x \in \mathbb{T}^n} \text{co} \, G (x)}=:K,
\end{equation}
where  with $\text{co} \, G (x)$ we denote the convex hull in $\mathbb{R}^n$ of the set $G (x)$,
and
$$
G ( \overline{x} ) := \text{supp}\, \mu \cap \{ (x,p) 
\in \mathbb{T}^n \times \mathbb{R}^n : x = \overline{x} \}, \quad \overline{x} \in \mathbb{T}^n.
$$
\end{prop}

\begin{rem}
We stress that the convex hull of the set $G (x)$ is taken only with respect to the variable $p$,
while the closure in the right-hand side of \eqref{Suppm} 
is taken in \textit{all} $\mathbb{T}^n \times \mathbb{R}^n$.
\end{rem}

\begin{proof}[Sketch of the proof]
For $\tau>0$ sufficiently small, we can choose an open set 
$K_\tau$ in $\mathbb{T}^n \times \mathbb{R}^n$ such that 
$K \subset K_\tau$, $\text{dist} \, (\partial K_{\tau}, K) < \tau$, 
and $K_\tau (x):= \{p\in \mathbb R^n:~(x,p) \in K_\tau\}$ is convex for every $x \in \mathbb T^n$.
 
Also, we can find a smooth open set $K_{2\tau} \subset \mathbb{T}^n \times \mathbb{R}^n$
such that, for every $x \in \mathbb{T}^n$,  
$K_{2 \tau} (x):= \{p\in \mathbb R^n:~(x,p) \in K_{2 \tau}\}$ is \textit{strictly convex},
$K_{2\tau} (x) \supset  K_\tau (x)$, and $\text{dist} \, (\partial K_{2\tau} (x), K_\tau (x)) < \tau$.

Finally, we can construct a smooth function 
$\eta_\tau  : \mathbb T^n \times \mathbb R^n \to \mathbb R$
such that for every $x \in \mathbb{T}^n$:
\begin{itemize}
\item $\eta_\tau(x,p)=0$ for $p \in K_\tau(x)$.
\item $p \mapsto \eta_\tau(x,p)$ is convex.
\item $p \mapsto \eta_\tau(x,p)$ is \textit{uniformly convex} on $\mathbb R^n \setminus K_{2\tau} (x)$. 
\end{itemize}
In this way, $\eta_\tau(x,p)=0$ on $K_\tau \supset K \supset \text{supp}\mu$. Therefore
$$
\int_{\mathbb T^n \times \mathbb R^n} \{\eta_\tau, H\} d\mu=0.
$$
Combining with \eqref{veryimp},
$$
\int_{\mathbb T^n \times \mathbb R^n} (\eta_\tau)_{p_k p_j} dm_{kj} =0,
$$
which implies $\text{supp}\,m \subset \bigcup_{x \in \mathbb T^n} K_{2\tau}(x)$. Letting $\tau \to 0$, we
finally get the desired result.
\end{proof}
As a consequence, we have the following corollary.
\begin{cor}
$$
\textnormal{supp}\,m \subset \overline{ \textnormal{co}  \{ H (x,p) \leq \overline{H} \} }.
$$

\end{cor}
 
\begin{proof}
The proof follows simply from the fact that for every $x \in \mathbb{T}^n$ we have
$$
G (x) \subset \{ H (x,p) \leq \overline{H} \}.
$$ 
\end{proof}
 
\end{section}

\begin{section}{Averaging}

In this section we prove some additional estimates concerning averaging
with respect to the process \eqref{statHJepsintro}. 
When necessary, to avoid confusion we will explicitly write the dependence on $P$.
Let us start with a definition.
\begin{defn}
We define the \textit{rotation number} $\rho_0$ associated to the measures $\mu$
and $m$ as
$$
\rho_0 : =\lim_{\varepsilon \to 0} 
\lim_{T \to + \infty} 
E \left[ \frac{\mathbf{x}^{\varepsilon} (T) - \mathbf{x}^{\varepsilon} (0) }{T}  \right],
$$
where the limit is taken along the same subsequences as in \eqref{limT} and \eqref{weakconv}.
\end{defn}
The following theorem gives a formula for the rotation number.

\begin{thm}

There holds 
\begin{equation} \label{rotnum}
\rho_0 = \int_{\mathbb{T}^n \times \mathbb{R}^n}  D_p H \, d \mu.
\end{equation}
Moreover, defining for every $\varepsilon > 0$ the variable
$\mathbf{X}^{\varepsilon} :=\mathbf{x}^{\varepsilon} +D_Pu^\varepsilon(\mathbf{x}^{\varepsilon})$,
we have
\begin{equation} \label{expv}
E \left[  \frac{\mathbf{X}^{\varepsilon} (T) -  \mathbf{X}^{\varepsilon} (0)}{T} \right]= - D_P \overline{H}^{\varepsilon}  (P),
\end{equation}
and
\begin{align*}
\lim_{T \to + \infty} 
E \left[  \frac{\left( \mathbf{X}^{\varepsilon} (T) -  \mathbf{X}^{\varepsilon} (0) + D_P \overline{H}^{\varepsilon}  (P)  T  \right)^2}{T} \right] 
&\leq 2 \, n \, \varepsilon^2  +  2 \int_{\mathbb{T}^n} | D_P u^{\varepsilon} |^2 \, d \theta_{\mu^{\varepsilon}} \\
&\hspace{.5cm}+ 2  \int_{\mathbb{T}^n} | D_p H - D_P \overline{H}^{\varepsilon} |^2 \, d \theta_{\mu^{\varepsilon}}.
\end{align*}

\end{thm}

\begin{proof}
Choosing $\phi (x) = x_i$ with $i=1,2,3$ in \eqref{exp} we obtain
$$
E \left[ \frac{\mathbf{x}^{\varepsilon} (T) - \mathbf{x}^{\varepsilon} (0) }{T}  \right]
= - E \left[  \frac{1}{T} \int_0^T D_p H 
(\mathbf{x}^{\varepsilon} (t) , \mathbf{p}^{\varepsilon} (t) ) \, dt \right].
$$
Passing to the limit as $T \to + \infty$ 
$$
\rho_{\varepsilon}:=
\lim_{T \to + \infty} 
E \left[ \frac{\mathbf{x}^{\varepsilon} (T) - \mathbf{x}^{\varepsilon} (0) }{T}  \right]
= \int_{\mathbb{T}^n \times \mathbb{R}^n}  D_p H \, d \mu^{\varepsilon}.
$$
We get \eqref{rotnum} by letting $\varepsilon$ go to zero. 

To prove \eqref{expv}, recalling It\^o's formula \eqref{Ito} we compute
\begin{align*}
d \mathbf{X}^{\varepsilon} 
& = d \mathbf{x}^{\varepsilon} + D^2_{P x} u^{\varepsilon} (\mathbf{x}^{\varepsilon}) \, d \mathbf{x}^{\varepsilon}
+ \frac{\varepsilon^2}{2} D_P \Delta u^{\varepsilon} (\mathbf{x}^{\varepsilon}) \, dt \\
& = \left( - D_p H ( \mathbf{x}^{\varepsilon}, \mathbf{p}^{\varepsilon}) ( I + D^2_{P x} u^{\varepsilon} (\mathbf{x}^{\varepsilon}) )
+ \frac{\varepsilon^2}{2} D_P \Delta u^{\varepsilon} (\mathbf{x}^{\varepsilon}) \right) dt 
+ \varepsilon ( I + D^2_{P x} u^{\varepsilon} (\mathbf{x}^{\varepsilon}) ) \, d w_t,
\end{align*}
where in the last equality we used \eqref{stoc}.
By differentiating equation \eqref{statHJeps} w.r.t. $P$ we obtain 
\begin{equation} \label{nice}
- D_p H ( \mathbf{x}^{\varepsilon}, \mathbf{p}^{\varepsilon}) ( I + D^2_{P x} u^{\varepsilon} (\mathbf{x}^{\varepsilon}) )
+ \frac{\varepsilon^2}{2} D_P \Delta u^{\varepsilon} (\mathbf{x}^{\varepsilon})
= - D_P \overline{H}^{\varepsilon} (P),
\end{equation}
so that 
\begin{equation} \label{dX}
d \mathbf{X}^{\varepsilon} 
= - D_P \overline{H}^{\varepsilon}  (P) \, dt + \varepsilon ( I + D^2_{P x} u^{\varepsilon}
(\mathbf{x}^{\varepsilon}) ) \, d w_t.
\end{equation}
Using the fact that
$$
E \left[ \int_0^T \varepsilon ( I + D^2_{P x} u^{\varepsilon} (\mathbf{x}^{\varepsilon}) ) \, d w_t \right] = 0,
$$
\eqref{expv} follows.

Finally, using once again It\^o's formula \eqref{Ito} and relation \eqref{dX} we can write
\begin{align*}
&d \left[ \left( \mathbf{X}^{\varepsilon} (t) -  \mathbf{X}^{\varepsilon} (0) + D_P \overline{H}^{\varepsilon}  (P)  t  \right)^2 \right] \\
&= 2 \left( \mathbf{X}^{\varepsilon} (t) -  \mathbf{X}^{\varepsilon} (0) + D_P \overline{H}^{\varepsilon}  (P)  t  \right)
(  d \, \mathbf{X}^{\varepsilon} +D_P \overline{H}^{\varepsilon}  (P)  \, dt ) 
+ \varepsilon^2 | I + D^2_{P x} u^{\varepsilon} (\mathbf{x}^{\varepsilon}) |^2 \, dt \\
& = 2 \, \varepsilon \left( \mathbf{X}^{\varepsilon} (t) -  \mathbf{X}^{\varepsilon} (0) + D_P \overline{H}^{\varepsilon}  (P)  t  \right)
 ( I + D^2_{P x} u^{\varepsilon} (\mathbf{x}^{\varepsilon}) ) \, d w_t 
 + \varepsilon^2 | I + D^2_{P x} u^{\varepsilon} (\mathbf{x}^{\varepsilon}) |^2 \, dt. 
\end{align*}
Hence, 
\begin{align*}
&E \left[  \left( \mathbf{X}^{\varepsilon} (T) -  \mathbf{X}^{\varepsilon} (0) + D_P \overline{H}^{\varepsilon}  (P)  T  \right)^2 \right] \\
&= E \left[  
\int_0^T 2 \, \varepsilon  \left( \mathbf{X}^{\varepsilon} (t) -  \mathbf{X}^{\varepsilon} (0) + D_P \overline{H}^{\varepsilon}  (P)  t  \right)
( I + D^2_{P x} u^{\varepsilon} (\mathbf{x}^{\varepsilon})) \, d w_t  
+ \int_0^T \varepsilon^2 | I + D^2_{P x} u^{\varepsilon} (\mathbf{x}^{\varepsilon}) |^2 \, dt \right] \\
&= E \left[  \int_0^T  \varepsilon^2 | I + D^2_{P x} u^{\varepsilon} (\mathbf{x}^{\varepsilon}) |^2 \, dt \right].
\end{align*}
Dividing by $T$ and letting $T$ go to infinity
\begin{align*}
&\lim_{T \to + \infty} 
E \left[  \frac{\left( \mathbf{X}^{\varepsilon} (T) -  \mathbf{X}^{\varepsilon} (0) + D_P \overline{H}^{\varepsilon}  (P)  T  \right)^2}{T} \right] 
= \lim_{T \to + \infty} E \left[  
\int_0^T  \frac{ \varepsilon^2  | I + D^2_{P x} u^{\varepsilon} (\mathbf{x}^{\varepsilon}) |^2}{T} \, dt \right] \\
& \hspace{.2cm}=  \varepsilon^2  \int_{\mathbb{T}^n} | I + D^2_{P x} u^{\varepsilon} |^2 \, d \theta_{\mu^{\varepsilon}} 
\leq 2 \, n \, \varepsilon^2 +  2 \, \varepsilon^2 \int_{\mathbb{T}^n}   | D^2_{P x} u^{\varepsilon} |^2 
\, d \theta_{\mu^{\varepsilon}} \\
&\hspace{.2cm}\leq 2 \, n \, \varepsilon^2 
+  2 \int_{\mathbb{T}^n} | D_P u^{\varepsilon} |^2 \, d \theta_{\mu^{\varepsilon}} 
+ 2  \int_{\mathbb{T}^n} | D_p H - D_P \overline{H}^{\varepsilon} |^2 \, d \theta_{\mu^{\varepsilon}},
\end{align*}
where we used \eqref{est2}.

\end{proof}
We conclude the section with a proposition which shows 
in a formal way how much relation \eqref{chvar}
is ``far'' from being an actual change of variables.
Let us set $w^{\varepsilon} (x,P) : = P \cdot x + u^{\varepsilon} (x,P)$, 
where $u^{\varepsilon} (x,P)$ is a $\mathbb{Z}^n$-periodic viscosity solution of \eqref{statHJepsintro},
and let $k \in \mathbb{Z}^n$.
We recall that in the convex setting the following weak version 
of the change of variables \eqref{chvar} holds \cite[Theorem~9.1]{EG}:
$$
\lim_{h \to 0} \int_{\mathbb{T}^n}  \Phi \left(  D^h_P u (x,P) \right) \, d \theta_{\mu}
= \int_{\mathbb{T}^n} \Phi \left( X \right) \, d X,
$$
for each continuous $\mathbb{Z}^n$-periodic function 
$\Phi: \mathbb{R}^n \to \mathbb{R}$, where 
$$
D^h_P u (x,P) := \left(  \frac{u (x,P + h e_1) - u (x,P)}{h} , \ldots, \frac{u (x,P + h e_n) - u (x,P)}{h}\right),
$$
$e_1,\ldots,e_n$ being the vectors of the canonical basis in $\mathbb{R}^n$.
The quoted result was proven by the authors by  considering the Fourier series of $\Phi$, 
and then analyzing the integral on the left-hand side mode by mode.
The next proposition shows what happens for a fixed mode in the non convex case.
\begin{prop}
The following inequality holds:
\begin{align*}
&( k \cdot D_P \overline{H}^{\varepsilon} ) \int_{\mathbb{T}^n} 
e^{ 2 \pi i k \cdot D_P w^{\varepsilon} }  \, d \theta_{\mu^{\varepsilon}} \\
&\hspace{1cm}\leq 2 \pi | k |^2 \left( \varepsilon^2
 +  \int_{\mathbb{T}^n} | D_P u^{\varepsilon} |^2 \, d \theta_{\mu^{\varepsilon}} 
+ \int_{\mathbb{T}^n} | D_p H - D_P \overline{H}^{\varepsilon} |^2 \, d \theta_{\mu^{\varepsilon}} \right).
\end{align*}
\end{prop}

\begin{proof}
Recalling identity \eqref{po} with
$$
\varphi (x) = e^{ 2 \pi i k \cdot D_P w^{\varepsilon} (x,P)}
$$ 
we obtain
\begin{align*}
0 &= \int_{\mathbb{T}^n} L^{\varepsilon, P} e^{ 2 \pi i k \cdot D_P w^{\varepsilon} } \, d \theta_{\mu^{\varepsilon}} \\
& = 2 \pi i \int_{\mathbb{T}^n} e^{ 2 \pi i k \cdot D_P w^{\varepsilon} }
\left[  L^{\varepsilon, P} \left( k \cdot D_P w^{\varepsilon} \right)
- \pi i \varepsilon^2 | D_x ( k \cdot D_P w^{\varepsilon}  ) |^2 \right] \, d \theta_{\mu^{\varepsilon}} \\
&= 2 \pi i \int_{\mathbb{T}^n} e^{ 2 \pi i k \cdot D_P w^{\varepsilon} }
\left[   k \cdot D_P \overline{H}^{\varepsilon} 
- \pi i \varepsilon^2 | D_x ( k \cdot D_P w^{\varepsilon}  ) |^2 \right] \, d \theta_{\mu^{\varepsilon}},
\end{align*}
where we used \eqref{dbeta} and the fact that $w^{\varepsilon} = P \cdot x + u^{\varepsilon}$.
Thus, thanks to estimate \eqref{est2}
\begin{align*}
&\left| ( k \cdot D_P \overline{H}^{\varepsilon} ) \int_{\mathbb{T}^n} 
e^{ 2 \pi i k \cdot D_P w^{\varepsilon} }  \, d \theta_{\mu^{\varepsilon}} \right| 
 \leq \pi \varepsilon^2 \int_{\mathbb{T}^n} | D_x ( k \cdot D_P w^{\varepsilon}  ) |^2 \, d \theta_{\mu^{\varepsilon}} \\
 & \leq 2 \pi  | k |^2 
 \left( \varepsilon^2 + \varepsilon^2 \int_{\mathbb{T}^n} | D^2_{P x} u^{\varepsilon} |^2 \, d \theta_{\mu^{\varepsilon}} \right) \\
 & \leq 2 \pi | k |^2 \left( \varepsilon^2
 +  \int_{\mathbb{T}^n} | D_P u^{\varepsilon} |^2 \, d \theta_{\mu^{\varepsilon}} 
+ \int_{\mathbb{T}^n} | D_p H - D_P \overline{H}^{\varepsilon} |^2 \, d \theta_{\mu^{\varepsilon}} \right).
 \end{align*}
\end{proof}

\begin{rem}
 When $H$ is uniformly convex, thanks to \eqref{unifcx2} the last chain of inequalities becomes
$$
\left| ( k \cdot D_P \overline{H}^{\varepsilon} ) \int_{\mathbb{T}^n} 
e^{ 2 \pi i k \cdot D_P w^{\varepsilon} }  \, d \theta_{\mu^{\varepsilon}} \right| 
\leq C | k |^2 \varepsilon^2 
\big( 1 + \textnormal{trace} \, ( D^2_{PP} \overline{H}^{\varepsilon} ) \big).
$$ 
Thus, if $\textnormal{trace} \, ( D^2_{PP} \overline{H}^{\varepsilon} ) \leq C$,
the right-hand side vanishes in the limit as $\varepsilon \to 0$, 
and we recover \cite[Theorem 9.1]{EG}.

\end{rem}

\end{section}

\begin{section}{Compensated compactness} \label{comcom}
In this section, some analogs of compensated compactness and Div-Curl lemma 
introduced by Murat and Tartar in the context of conservation laws (see \cite{Ev2}, \cite{T}) 
will be studied, in order to better understand the support of the Mather measure $\mu$. 
Similar analogs are also considered in \cite{Ev}, 
to investigate the shock nature of non-convex Hamilton-Jacobi equations.\\ 
What we are doing here is quite different from the original Murat and Tartar work (see \cite{T}), 
since we work on the support of the measure $\theta_{\mu^\varepsilon}$. 
Besides, our methods work on arbitrary dimensional space $\mathbb R^n$ 
while usual compensated compactness and Div-Curl lemma in the context of conservation laws 
can only deal with the case $n=1, 2$.
However, we can only derive one single relation and this is not enough to characterize the support of $\mu$ as
in the convex case.
To avoid confusion, when necessary we will explicitly write the dependence on the $P$ variable.

Let $\phi$ be a smooth function from $\mathbb T^n \times \mathbb R^n \rightarrow \mathbb R$, and
let $\rho^\varepsilon=\{\phi,H\} \theta_{\mu^\varepsilon} 
+\dfrac{ \varepsilon^2}{2} \phi_{p_j p_k} u^\varepsilon_{x_i x_j} u^\varepsilon_{x_i
x_k}\theta_{\mu^\varepsilon}$.
By (\ref{imp}) and (\ref{est}), there exists $C>0$ such that
$$
\int_{\mathbb T^n} |\rho^\varepsilon| dx \le C.
$$
So, up to passing to some subsequence, if necessary, we may assume that 
$\rho^\varepsilon \stackrel{*}{\rightharpoonup} \rho$ as a (signed) measure.\\
By (\ref{veryimp}), $\rho(\mathbb T^n)=0$. We have the following theorem.
\begin{thm}
\label{comp}
The following properties are satisfied:
\begin{itemize}
\item[(i)] for every $\phi \in C (\mathbb{T}^n \times \mathbb{R}^n)$
\begin{equation}
\label{comp1}
\int_{\mathbb T^n \times \mathbb R^n} D_pH \cdot (p-P) \, \phi(x,p) \, d\mu = \int_{\mathbb T^n} u \, d\rho;
\end{equation}
\item[(ii)] for every $\phi \in C (\mathbb{T}^n \times \mathbb{R}^n)$ 
and for every $\eta \in C^1(\mathbb T^n)$,
\begin{equation}
\label{comp2}
\int_{\mathbb T^n \times \mathbb R^n} D_pH \cdot D\eta \, \phi(x,p) \, d\mu = \int_{\mathbb T^n} \eta d\rho.
\end{equation}
\end{itemize}
\end{thm}

\begin{proof}
Let $w^\varepsilon=\phi(x,P+D_x u^\varepsilon)$. Notice first that
$$
\int_{\mathbb T^n \times \mathbb R^n} D_p H \cdot (p-P) \, \phi(x,p) \, d\mu 
= \lim_{\varepsilon \rightarrow 0} \int_{\mathbb T^n} D_p H(x,P+D_x u^\varepsilon) 
\cdot D_x u^\varepsilon w^\varepsilon d \theta_{\mu^\varepsilon}.
$$
Integrating by parts the right hand side of the above equality we obtain
\begin{align}
&\int_{\mathbb T^n} D_p H(x,P+D_x u^\varepsilon) \cdot D_x u^\varepsilon w^\varepsilon d
\theta_{\mu^\varepsilon}=-\int_{\mathbb T^n} u^\varepsilon \text{div}(D_pH w^\varepsilon
\theta_{\mu^\varepsilon}) dx \notag\\
=&- \int_{\mathbb T^n} u^\varepsilon(
\text{div}(D_pH\theta_{\mu^\varepsilon})w^\varepsilon+D_pH \cdot D_x w^\varepsilon 
\theta_{\mu^\varepsilon}) dx
=\int_{\mathbb T^n} u^\varepsilon (\dfrac{\varepsilon^2}{2} 
\Delta \theta_{\mu^\varepsilon} w^\varepsilon 
- D_p H \cdot D_x w^\varepsilon \theta_{\mu^\varepsilon})dx. \notag
\end{align}
After several computations, by using (\ref{statHJeps}) we get
$$
D_pH \cdot D_x w^\varepsilon=-\{\phi,H\} + \dfrac{\varepsilon^2}{2} \phi_{p_i} \Delta u^\varepsilon_{x_i}.
$$
Hence
\begin{align}
&\dfrac{\varepsilon^2}{2} \Delta \theta_{\mu^\varepsilon} w^\varepsilon
- D_p H \cdot D_x w^\varepsilon \theta_{\mu^\varepsilon}
=\dfrac{\varepsilon^2}{2} \Delta \theta_{\mu^\varepsilon} w^\varepsilon+\{\phi,H\}{\theta_{\mu^\varepsilon}} 
- \dfrac{\varepsilon^2}{2} \phi_{p_i} \Delta u^\varepsilon_{x_i}{\theta_{\mu^\varepsilon}} \notag\\
=&\dfrac{\varepsilon^2}{2} \Delta w^\varepsilon
\theta_{\mu^\varepsilon}+\dfrac{\varepsilon^2}{2}(\text{div}(D_x \theta_{\mu^\varepsilon} w^\varepsilon)
-\text{div}(D_x w^\varepsilon \theta_{\mu^\varepsilon}))+\{\phi,H\}{\theta_{\mu^\varepsilon}} 
- \dfrac{\varepsilon^2}{2} \phi_{p_i} \Delta u^\varepsilon_{x_i} {\theta_{\mu^\varepsilon}} \notag\\
=&\dfrac{\varepsilon^2}{2} (\phi_{p_j p_k} u^\varepsilon_{x_i x_j} u^\varepsilon_{x_i x_k}
{+ \phi_{p_j x_i} u^{\varepsilon}_{x_j x_i}}
+\phi_{x_i x_i} +\phi_{p_i} \Delta u^\varepsilon_{x_i}) \theta_{\mu^\varepsilon}  \notag\\
&\hspace{.4cm}+\dfrac{\varepsilon^2}{2}(\text{div}(D_x \theta_{\mu^\varepsilon} w^\varepsilon)
-\text{div}(D_x w^\varepsilon \theta_{\mu^\varepsilon})) +\{\phi,H\} {\theta_{\mu^\varepsilon}}  
- \dfrac{\varepsilon^2}{2} \phi_{p_i} \Delta u^\varepsilon_{x_i} {\theta_{\mu^\varepsilon}}  \notag\\
=& \rho^\varepsilon +\dfrac{\varepsilon^2}{2} \phi_{x_i x_i} \theta_{\mu^\varepsilon} 
{+ \frac{\varepsilon^2}{2} \phi_{p_j x_i} u^{\varepsilon}_{x_j x_i}} 
+ \dfrac{\varepsilon^2}{2}(\text{div}(D_x \theta_{\mu^\varepsilon} w^\varepsilon)
-\text{div}(D_x w^\varepsilon \theta_{\mu^\varepsilon})). \notag
\end{align}
Therefore
\begin{align}
&\int_{\mathbb T^n \times \mathbb R^n} D_pH \cdot (p-P) \, \phi(x,p) \, d\mu \notag \\
&\hspace{.2cm}= \lim_{\varepsilon \rightarrow 0}  \int_{\mathbb T^n} u^\varepsilon
\left[ \rho^\varepsilon +\dfrac{\varepsilon^2}{2} \phi_{x_i x_i} \theta_{\mu^\varepsilon} 
{+ \frac{\varepsilon^2}{2} \phi_{p_j x_i} u^{\varepsilon}_{x_j x_i}} 
+ \dfrac{\varepsilon^2}{2}(\text{div}(D_x \theta_{\mu^\varepsilon} w^\varepsilon)
-\text{div}(D_x w^\varepsilon \theta_{\mu^\varepsilon})) \right] dx. \label{comp_eps}
\end{align}
Since $u^\varepsilon$ converges uniformly to $u$,
$$
\lim_{\varepsilon \rightarrow 0}  \int_{\mathbb T^n}u^\varepsilon \rho^\varepsilon dx 
=\int_{\mathbb T^n} u \, d\rho.
$$
The second term in the right hand side of (\ref{comp_eps}) obviously converges to $0$ as $\varepsilon
\rightarrow 0$. 
The third term also tends to $0$ by (\ref{est}).\\
Let's look at the last term.
We have
\begin{align}
& \left|\lim_{\varepsilon \rightarrow 0}\dfrac{\varepsilon^2}{2} 
\int_{\mathbb T^n} u^\varepsilon (\text{div} (D_x \theta_{\mu^\varepsilon} w^\varepsilon)
{-}\text{div} (D_x w^\varepsilon \theta_{\mu^\varepsilon} )) dx \right| 
= \left| \lim_{\varepsilon \rightarrow 0}\dfrac{\varepsilon^2}{2} \int_{\mathbb T^n} 
{-} D_x u^\varepsilon \cdot D_x \theta_{\mu^\varepsilon} w^\varepsilon
+D_x u^\varepsilon \cdot D_xw^\varepsilon \theta_{\mu^\varepsilon} dx \right| \notag \\
=& \left| \lim_{\varepsilon \rightarrow 0} \dfrac{\varepsilon^2}{2}\int_{\mathbb T^n} 
\text{div}(D_x u^\varepsilon w^\varepsilon)\theta_{\mu^\varepsilon}
+D_x u^\varepsilon \cdot D_x w^\varepsilon \theta_{\mu^\varepsilon} dx \right|
= \left|\lim_{\varepsilon \rightarrow 0} \dfrac{\varepsilon^2}{2}
\int_{\mathbb T^n} (\Delta u^\varepsilon w^\varepsilon
+ 2{D_x u^\varepsilon \cdot D_x w^\varepsilon} ) \theta_{\mu^\varepsilon} dx \right|  \notag\\
\le & \lim_{\varepsilon \rightarrow 0} C\varepsilon^2 \int_{\mathbb T^n} |D^2_{xx} u^\varepsilon|
\theta_{\mu^\varepsilon} dx 
\le \lim_{\varepsilon \rightarrow 0} C\varepsilon =0, \notag
\end{align}
which implies (\ref{comp1}). Relation (\ref{comp2}) can be derived similarly.
\end{proof}
As a consequence, we have the following corollary.
\begin{cor} \label{rem_comp}
Let $u ( \cdot, P)$ be a classical solution of \eqref{statHJ}, 
and let $\mu$ be the corresponding Mather measure given by Theorem \ref{exist}.
Then,     
$$
D_pH \cdot (p-P-D_x u)=0 \quad \text{ in supp}\, \mu.
$$

\end{cor}

\begin{proof}
By (\ref{comp1}) and (\ref{comp2})
$$
\int_{\mathbb T^n} D_pH \cdot (p-P-D_x u) \, \phi \, d\mu =0,
$$
for all $\phi$. Therefore, the conclusion follows. 
\end{proof}

\end{section}

\begin{section}{Examples} \label{Ex}

In this section, we study non-trivial examples where the Mather measure $\mu$ is invariant under the Hamiltonian dynamics. 
Notice that, by \eqref{veryimp}, the Mather measure $\mu$ is invariant under the Hamiltonian dynamics 
if and only if the dissipation measures $(m_{kj})$ vanish.
An example in Section \ref{counter} shows that this is not always guaranteed.
As explained in \cite{Ev}, the dissipation measures $m_{kj}$
record the jump of the gradient $D_x u$ along the shock lines.

We investigate now under which conditions we still have the invariance property (1).
We provide some partial answers by studying several examples,
which include the important class of strongly quasiconvex Hamiltonians (see \cite{F5}).

\begin{subsection}{$H$ is uniformly convex} 

There exists $\alpha>0$ so that $D_{pp} ^2 H \ge \alpha>0$.\\
Let $\lambda=0$ in (\ref{iul}) then
$$
0=\int_{\mathbb T^n \times \mathbb R^n} H_{p_k p_j} dm_{kj},
$$
which implies $m_{kj}=0$ for all $1 \le k,j \le n$.
We then can follow the same steps as in \cite{EG} to get that $\mu$ also satisfies (2).

\end{subsection}

\begin{subsection}{Uniformly convex conservation law}

Suppose that there exists $F(p,x)$, strictly convex in $p$, such that $\{F,H\}$=0. Then $m=0$. 

\end{subsection}

\begin{subsection}{Some special non-convex cases}
The cases we consider here are somehow variants of the uniformly convex case.

Suppose there exists $\phi$ uniformly convex and a smooth real function $f$ 
such that either $\phi=f(H)$ or $H=f(\phi)$.
Then, by (\ref{veryimp}) we have $m_{kj}=0$ for all $k,j$.
In particular, if $H=f(\phi)$ with $f$ increasing, then $H$ is quasiconvex.\\
One explicit example of the above variants is $H(x,p) = (|p|^2+V(x))^2$, 
where $V: \mathbb T^n \rightarrow \mathbb R$ is smooth and may take negative values.
Then $H(x,p)$ is not convex in $p$ anymore.
Anyway, we can choose $\phi(x,p)=|p|^2+V(x)$, so that $H(x,p)=(\phi(x,p))^2$ and $\phi$ 
is uniformly convex in $p$.
Therefore, $\mu$ is invariant under the Hamiltonian dynamics.\\

\end{subsection}

\begin{subsection} {The case when $n=1$}

Let's consider the case $H(x,p)=H(p)+V(x)$.\\
In this particular case, property (H3) implies that $|H(x,p)| \to \infty$ as $|p| \to + \infty$.
Let us suppose that 
$$
\lim_{|p| \to + \infty} H (p) = + \infty.
$$
Assume also that there exists $p_0 \in \mathbb R$ such that $H'(p)=0$ 
if and only if $p=p_0$ and $H''(p_0) \ne 0$. 
Notice that $H(p)$ does not need to be convex.
Obviously, uniform convexity of $H$ implies this condition.\\
We will show that $m_{11}=0$, which implies that $\mu$ is invariant under the Hamiltonian dynamics.
From our assumptions, we have that 
$H'(p)>0$ for $p >p_0$, $H'(p)<0$ for $p<p_0$ and hence $H''(p_0)>0$.
Then there exists a neighborhood $(p_0 - r, p_0+r)$ of $p_0$ such that
$$
H''(p) > \dfrac{H''(p_0)}{2}, \hspace{1cm} \forall~p \in (p_0 - r, p_0+r).
$$
And since the support of $m_{11}$ is bounded, we may assume
$$
\text{supp}(m_{11}) \subset \mathbb T \times [-M,M], 
$$
for some $M>0$ large enough. We can choose $M$ large so that  $(p_0-r,p_0+r) \subset (-M,M)$.\\
Since $|H'(p)|^2 >0$ for $p \in [-M,M] \setminus (p_0-r,p_0+r)$ and $[-M,M] \setminus (p_0-r,p_0+r)$ is compact, there exists $\gamma>0$ such that
$$
|H'(p)|^2 \ge \gamma>0,\hspace{1cm} \forall~p \in [-M,M] \setminus (p_0-r,p_0+r).
$$
Hence, by choosing $\lambda \gg 0$
$$
\lambda |H'(p)|^2 + H''(p) \ge \dfrac{H''(p_0)}{2}, \hspace{1cm} \forall~ p \in [-M,M],
$$
which shows $m_{11}=0$ by (\ref{iul}).
\end{subsection}

\begin{subsection} {Case in which there are more conserved quantities}
Let's consider
$$
H(x,p)=H(p)+V(x_1+...+x_n), 
$$
where $V: \mathbb T \rightarrow \mathbb R$ is smooth.\\
For $k \ne j$, define $\Phi^{kj}=p_k-p_j$. 
It is easy to see that $\{H, \Phi^{kj}\}=0$ for any $k \ne j$.\\
Therefore $\{H, (\Phi^{kj})^2\}=0$ for any $k \ne j$.\\
For fixed $k \ne j$, let $\phi= (\Phi^{kj})^2$ in (\ref{veryimp}) then
$$
2\int_{\mathbb T^n \times \mathbb R^n}  (m_{kk}-2m_{kj}+m_{jj}) \, dxdp=0.
$$
The matrix of {\it dissipation measures} $(m_{kj})$ is non-negative definite, 
therefore $m_{kk}-2m_{kj}+m_{jj} \ge 0$. Thus, $m_{kk}-2m_{kj}+m_{jj}=0$ for any $k \ne j$.\\
Let $\varepsilon \in (0,1)$ and take $\xi=(\xi_1,...,\xi_n)$,
where $\xi_k=1+\varepsilon, \xi_j=-1$ and $\xi_i=0$ otherwise. We have
$$
0 \le m_{kj} \xi_k \xi_j=(1+\varepsilon)^2 m_{kk} -2(1+\varepsilon) m_{kj}
+m_{jj}=2\varepsilon(m_{kk}-m_{kj})+\varepsilon^2 m_{kk}.
$$
Dividing both sides of the inequality above by $\varepsilon$ and letting $\varepsilon \rightarrow 0$,
$$
m_{kk}-m_{kj} \ge 0.
$$
Similarly, $m_{jj}-m_{kj} \ge 0$. 
Thus, $m_{kk}-m_{kj}=m_{jj}-m_{kj}=0$ for all $k \ne j$.\\
Hence, there exists a non-negative measure $m$ such that
$$
m_{kj}=m \ge 0, \hspace{1cm} \forall~k,j.
$$
Therefore, (\ref{iul}) becomes
$$
0=\int_{\mathbb T^n \times \mathbb R^n} e^{\lambda H} 
\Big(\lambda \big(\sum_j H_{p_j} \big)^2 + \sum_{j,k} H_{p_j p_k} \Big) \, dm.
$$
We here point out two cases which guarantee that $m=0$. In the first case, 
assuming additionally that $H(p)=H_1(p_1)+...H_n(p_n)$ and $H_2,...,H_n$ are convex, 
but not necessarily uniformly convex (their graphs may have flat regions) 
and $H_1$ is uniformly convex, then we still have $m=0$.

In the second case, suppose that $H(p)=H(|p|)$, where $H:[0,\infty) \rightarrow \mathbb R$ is smooth,
$H'(0)=0, H''(0)>0$ and $H'(s)>0$ for $s>0$.
Notice that $H$ is not necessarily convex.
This example is similar to the example above when $n=1$.
Then for $p \ne 0$
$$
\lambda \big(\sum_j H_{p_j} \big)^2 + \sum_{j,k} H_{p_j p_k}=n
\dfrac{H'}{|p|}+\dfrac{(p_1+...+p_n)^2}{|p|^2} \left( \lambda (H')^2+H''-\dfrac{H'}{|p|} \right),
$$
and at $p=0$
$$
\lambda \big(\sum_j H_{p_j}(0) \big)^2 + \sum_{j,k} H_{p_j p_k}(0)=nH''(0)>0.
$$
\end{subsection}
So, we can choose $r>0$, small enough, so that for $|p| <r$
$$
\lambda \big( \sum_j H_{p_j} \big)^2 + \sum_{j,k} H_{p_j p_k}> \dfrac{n}{2}H''(0)>0.
$$
Since the support of $m$ is bounded, there exists $M>0$ large enough
$$
\text{supp} \, m \subset \mathbb T^n \times \{ p: |p| \le M\}.
$$
Since $\min_{s \in [r,M]} H'(s) >0$, by choosing $\lambda \gg 0$, we finally have for $|p| \le M$
$$
\lambda \big( \sum_j H_{p_j} \big)^2 + \sum_{j,k} H_{p_j p_k} \ge \beta>0,
$$
for $\beta =\dfrac{n}{2} \min \left\{ H''(0), \dfrac{\min_{s \in [r,M]} H'(s) }{M} \right\}$.\\
Thus $m=0$, and therefore $\mu$ is invariant under the Hamiltonian dynamics.

\begin{subsection}{Quasiconvex Hamiltonians: a special case} \label{spcase}
Let's consider
$$
H(x,p)=H(|p|)+V(x),
$$
where $H:[0,\infty) \rightarrow \mathbb R$ is smooth, $H'(0)=0, H''(0)>0$ and $H'(s)>0$ for $s>0$.\\
Once again, notice that $H$ is not necessarily convex.
We here will show that $(m_{jk})=0$. 
For $p \ne 0$ then
$$
(\lambda H_{p_j} H_{p_k}+H_{p_j p_k}) m_{jk}=\dfrac{H'}{|p|} (m_{11}+...+m_{nn})
+ \left( \lambda (H')^2 +H''-\dfrac{H'}{|p|} \right) \dfrac{p_j p_k m_{jk}}{|p|^2}.
$$
For any symmetric, non-negative definite matrix $m=(m_{jk})$ we have the following inequality
$$
0 \le p_j p_k m_{jk} \le |p|^2 \, \text{trace} \, m = |p|^2(m_{11}+...+m_{nn}).
$$
There exists $r>0$ small enough so that for $|p|<r$
$$
\dfrac{H'}{|p|}>\dfrac{3}{4} H''(0); \quad \quad 
\left| \dfrac{H'}{|p|} - H'' \right| <\dfrac{1}{4} H''(0).
$$
Hence for $|p|<r$
$$
(\lambda H_{p_j} H_{p_k}+H_{p_j p_k}) m_{jk} \ge \dfrac{1}{2} H''(0) (m_{11}+...+m_{nn}).
$$
Since the support of $(m_{jk})$ is bounded, there exists $M>0$ large enough
$$
\text{supp} \, m_{jk} \subset \mathbb T^n \times \{ p: |p| \le M\}, \hspace{1cm} \forall~j,k.
$$
Since $\min_{s \in [r,M]} H'(s)>0$, by choosing $\lambda \gg 0$ we finally have for $|p| \le M$
$$
(\lambda H_{p_j} H_{p_k}+H_{p_j p_k}) m_{jk} \ge \beta(m_{11}+...+m_{nn}),
$$
for $\beta =\min \left\{ \dfrac{H''(0)}{2}, \dfrac{\min_{s \in [r,M]} H'(s) }{M} \right\}>0$.\\
We then must have $m_{11}+...m_{nn}=0$, which implies $(m_{jk})=0$.
Thus, $\mu$ is invariant under the Hamiltonian dynamics in this case.\\
We now derive the property (2) of $\mu$ rigorously. 
Since the support of $\mu$ is also bounded, we can use a similar procedure as above to show that
$\phi(x,p)=e^{\lambda H(x,p)}$ is uniformly convex in $\mathbb T^n \times \bar B(0,M) \supset
\text{supp}(\mu)$ for some $\lambda$ large enough.\\
More precisely,
$$
\phi_{p_j p_k} \xi_j \xi_k \ge e^{\lambda H} \beta |\xi|^2, \quad \xi \in \mathbb R^n,~(x,p) \in \mathbb T^n \times \bar B(0,M),
$$
for $\beta$ chosen as above. Then doing the same steps as in \cite{EG}, we get $\mu$ satisfies (2).\\
There is another simple approach to prove (2) by using the properties we get in this non-convex setting.
Let's just assume that $u$ is $C^1$ on the support of $\mu$.\\
By Remark \ref{rem_comp}, it follows that $D_pH.(p-P-Du)=0$ on support of $\mu$. 
And since $D_pH(x,p)=H'(|p|) \dfrac{p}{|p|}$ for $p \ne 0$ and $H'(|p|)>0$, we then have $p.(p-P-Du)=0$ on support of $\mu$. 
Hence $|p|^2 = p.(P+Du)$ on $\text{supp}(\mu)$.\\
Besides, $H(x,p)=H(x,P+Du(x)) = \overline{H}(P)$ on $\text{supp}(\mu)$ by property (a) of Mather measure and the assumption that $u$ is $C^1$ on $\text{supp}(\mu)$. It follows that $H(|p|)=H(|P+Du|)$. 
Therefore, $|p|=|P+Du|$ by the fact that $H(s)$ is strictly increasing.\\
So we have $|p|^2 = p.(P+Du)$ and $|p|=|P+Du|$ on $\text{supp}(\mu)$, which implies $p=P+Du$ on $\text{supp}(\mu)$, which is the property (2) of $\mu$.\\
\end{subsection}

\begin{subsection}{Quasiconvex Hamiltonians} \label{qcx}

We treat now the general case of uniformly quasiconvex Hamiltonians.
We start with a definition.
\begin{defn}
A smooth set $A \subset \mathbb{R}^n$ is said to be \textit{strongly convex with convexity constant $c$}
if there exists a positive constant $c$ with the following property.
For every $p \in \partial A$ there exists an orthogonal coordinate system $(q_1, \dots, q_n)$ centered at $p$,
and a coordinate rectangle $R = (a_1, b_1){\times}\dots{\times}(a_n, b_n)$ containing $p$
such that $T_p \partial A = \{ q_n =0 \}$ and
$A \cap R \subset \{q \in R : c \sum_{i=1}^{n-1} |q_i|^2 \leq q_n \leq b_n \}$.
\end{defn}
The previous definition can be stated in the following equivalent way, 
by requiring that for every $p \in \partial A$ 
$$
\left( \mathbf{B}_p \mathbf{v} \right) \cdot \mathbf{v} \geq c |\mathbf{v} |^2 \quad \text{ for every }\mathbf{v} \in T_p \partial A,
$$
where $\mathbf{B}_p : T_p \partial A \times T_p \partial A \to \mathbb{R}$ is 
the second fundamental form of $\partial A$ at $p$.

We consider in this subsection strongly quasiconvex Hamiltonians.
That is, we assume that there exists $c >0$ such that
\begin{itemize}

\item[(j)] $\{ p \in \mathbb{T}^n : H(x,p) \leq a \}$ is strongly convex with convexity constant $c$
for every $a \in \mathbb{R}$ and for every $x \in \mathbb{T}^n$.

\end{itemize}
In addition, we suppose that there exists $\alpha \in \mathbb{R}$ such that for every $x \in \mathbb{T}^n$
\begin{itemize}

\item[(jj)] There exists unique $\overline{p} \in \mathbb{R}^n$ s.t. $D_p H (x, \overline{p}) = 0$, and 
$$
D^2_{pp} H (x, \overline{p}) \geq \alpha. 
$$

\end{itemize}
Notice that the special case presented in Section \ref{spcase},
where the level sets are spheres, fits into this definition.
We will show that under hypotheses (j)--(jj) there exists $\lambda > 0$ such that
$$
\lambda \, D_p H \otimes D_p H + D^2_{pp} H \quad \text{ is positive definite. }
$$
From this, thanks to relation \eqref{iul}, we conclude that $m_{kj} = 0$.
First, we state a well-known result.
We give the proof below, for the convenience of the reader. 
\begin{prop}
Let (j)--(jj) be satisfied, and let $(x^*,p^*) \in \mathbb{T}^n \times \mathbb{R}^n$ be such that $D_p H (x^*,p^*) \neq 0$.
Then
\begin{equation} \label{diffgeo}
D_p H (x^*,p^*) \perp T_{p^*} \mathcal{C} \quad \text{ and }\quad D^2_{pp} H (x^*,p^*) = |D_p H (x^*,p^*)| \mathbf{B}_{p^*},
\end{equation}
where $\mathbf{B}_{p^*}$ denotes the second fundamental form of the level set
$$
\mathcal{C}:= \{ p \in \mathbb{R}^n : H (x^*,p) = H (x^*,p^*) \}
$$ 
at the point $p^*$.
\end{prop}

\begin{proof}
By the smoothness of $H$, there exists a neighborhood $U \subset \mathbb{R}^n$ of $p^*$ 
and $n$ smooth functions $\nu : U \to \mathcal{S}^{n-1}$,
$\tau_i : U \to \mathcal{S}^{n-1}$, $i=1,\ldots,n-1$, 
such that for every $p \in U$ the vectors $\{ \tau_1 (p), \ldots, \tau_{n-1} (p) , \nu (p) \}$ 
are a smooth orthonormal basis of $\mathbb{R}^n$, 
and for every $p \in U \cap \mathcal{C}$ $\tau_1 (p), \ldots, \tau_{n-1} (p) \in T_p \mathcal{C}$.
Let now $i,j \in \{ 1, \ldots, n-1 \}$ be fixed. 
Since 
$$
H(x^*, p ) = a \quad \forall \, p \in U,
$$
differentiating w.r.t $\tau_i (p)$ we have
\begin{equation} \label{DH}
D_p H (x^*, p) \cdot \tau_i (p) = 0 \quad \forall \, p \in U \cap \mathcal{C}.
\end{equation}
Computing last relation at $p=p^*$ we get that $D_p H (x^*,p^*) \perp T_{p^*} \mathcal{C}$.
Differentiating \eqref{DH} along the direction $\tau_j (p)$ and computing at $p=p^*$
\begin{equation} \label{uj}
\left( D^2_{pp} H (x^*, p^*) \tau_j (p^*) \right) \cdot \tau_i (p^*)  + D_p H (x^*, p^*) \cdot \left( D_p \tau_i (p^*) \tau_j (p^*) \right)  = 0.
\end{equation}
Notice that by differentiating along the direction $\tau_{j} (p)$ the identity $\tau_{i} (p) \cdot \nu (p) = 0$
and computing at $p^*$ we get
$$
\left( D_p \tau_i (p^*) \tau_j (p^*) \right) \cdot \nu (p^*)
= - \left( D_p \nu (p^*) \tau_j (p^*) \right) \cdot \tau_i (p^*).
$$
Plugging last relation into \eqref{uj}, and choosing $\nu (p^*)$ oriented in the direction of $D_p H (x^*, p^*)$ we have 
\begin{align*}
&\left( D^2_{pp} H (x^*, p^*) \tau_j (p^*) \right) \cdot \tau_i (p^*)  
= - | D_p H (x^*, p^*) | \left( D_p \tau_i (p^*) \tau_j (p^*) \right) \cdot \nu (p^*) \\
&= | D_p H (x^*, p^*) | \left( D_p \nu (p^*) \tau_j (p^*) \right) \cdot \tau_i (p^*)
= | D_p H (x^*, p^*) | \left( \mathbf{B}_{p^*} \tau_j (p^*) \right) \cdot \tau_i (p^*).
\end{align*}

\end{proof}

For every vector $v \in \mathbb{R}^n$, we consider the  decomposition
$$
v = v_{\parallel} \mathbf{v^{\parallel}} + v_{\perp} \mathbf{v^{\perp}}, 
$$
with $v^{\parallel}, v^{\perp} \in \mathbb{R}$, $| \mathbf{v^{\parallel}} | = | \mathbf{v^{\perp}} | = 1$, 
$\mathbf{v^{\parallel}} \in T_{p^*} \mathcal{C}$,
and $\mathbf{v^{\perp}} \in ( T_{p^*} \mathcal{C} )^{\perp}$.
By hypothesis (jj) and by the smoothness of $H$, there exist $\tau > 0$
and $\alpha' \in ( 0, \alpha )$, independent of $(x,p)$, such that
$$
D_{pp}^2 H (x,p) \geq \alpha' \quad 
\text{ for every } (x,p) \in \{ | D_p H | \leq \tau\}.
$$
Let us now consider two subcases:\\
\textbf{Case 1: $(x,p) \in \{ | D_p H | \leq \tau\}$} \\

First of all, notice that
$$
\lambda \, D_p H \otimes D_p H v \cdot v
= \lambda | D_p H \cdot v |^2 = \lambda \, v_{\perp}^2 \, | D_p H |^2.
$$
Then, we have
$$
( \lambda \, D_p H \otimes D_p H + D^2_{pp} H ) v \cdot v 
= \lambda \, v_{\perp}^2 \, | D_p H |^2 + ( D^2_{pp} H  v \cdot v ) \geq \alpha' |v|^2.
$$
\textbf{Case 2: $(x,p) \in \{ | D_p H | > \tau\}$} \\

In this case we have
$$
D^2_{pp} H \mathbf{v^{\parallel}} \cdot \mathbf{v^{\parallel}}
\geq c |D_p H|,
$$
which then yields
\begin{align*}
D^2_{pp} H v \cdot v
&= v_{\parallel}^2 ( D^2_{pp} H \mathbf{v^{\parallel}} \cdot \mathbf{v^{\parallel}} )
+ 2 v_{\parallel} v_{\perp} ( D^2_{pp} H \mathbf{v^{\parallel}} \cdot \mathbf{v^{\perp}} )
+ v_{\perp}^2 ( D^2_{pp} H \mathbf{v^{\perp}} \cdot \mathbf{v^{\perp}} ) \\
&\geq c \, v_{\parallel}^2 |D_p H| + 2 v_{\parallel} v_{\perp} ( D^2_{pp} H \mathbf{v^{\parallel}} \cdot \mathbf{v^{\perp}} )
+ v_{\perp}^2 ( D^2_{pp} H \mathbf{v^{\perp}} \cdot \mathbf{v^{\perp}} ).
\end{align*}
By \eqref{suppcompact} we have
$$
| D^2_{pp} H | \leq C \quad \text{ along }\text{supp}\, \mu.
$$
Thus, 
\begin{align*}
&( \lambda \, D_p H \otimes D_p H + D^2_{pp} H ) v \cdot v \\
&\geq \lambda \, v_{\perp}^2 \, | D_p H |^2
+ c \, v_{\parallel}^2 |D_p H| + 2 v_{\parallel} v_{\perp} ( D^2_{pp} H \mathbf{v^{\parallel}} \cdot \mathbf{v^{\perp}} )
+ v_{\perp}^2 ( D^2_{pp} H \mathbf{v^{\perp}} \cdot \mathbf{v^{\perp}} ) \\
&\hspace{1cm}\geq v_{\perp}^2 \left( \lambda \, | D_p H |^2 - C \, \right)
- 2 C | v_{\parallel} | | v_{\perp} | + c \, v_{\parallel}^2 |D_p H| \\
&\hspace{1cm}> v_{\perp}^2 \Big( \lambda \, \tau^2 - C \Big(1 + \frac{1}{\eta^2} \Big) \Big) 
+ v_{\parallel}^2 ( c \, \tau - C \eta^2 \, ) .
\end{align*}
Choosing first $\eta^2 < \frac{c \, \tau}{C}$, and then
$$
\lambda > \frac{C}{\tau^2} \Big(1 + \frac{1}{\eta^2} \Big),
$$
we obtain
$$
( \lambda \, D_p H \otimes D_p H + D^2_{pp} H ) v \cdot v \geq \alpha'' |v|^2,
$$
for some $\alpha'' > 0$, independent of $(x,p)$.

\textbf{General Case} \\

In the general case, we have
$$
( \lambda \, D_p H \otimes D_p H + D^2_{pp} H ) v \cdot v \geq \gamma |v|^2,
$$
where $\gamma : = \min \{ \alpha' , \alpha'' \}$.\\
Similar to the case above, we basically have that $\phi(x,p)=e^{\lambda H(x,p)}$ 
is uniformly convex on the support of $\mu$ for $\lambda$ large enough.
Hence, by repeating again the same steps as in \cite{EG}, 
we finally get that $\mu$ satisfies (2).
As already mentioned in the introduction, we observe that one could also
study the case of uniformly convex Hamiltonians by duality, that is, by considering a function 
$\Phi: \mathbb{R} \to \mathbb{R}$ such that $\Phi ( H (x, \cdot))$ is convex for each $x \in \mathbb{T}^n$.
In this way, the dynamics can be seen as a reparametrization of the dynamics
associated to the convex Hamiltonian $\Phi (H)$.

\end{subsection}

\end{section}

\begin{section}{A one dimensional example of nonvanishing dissipation measure $m$} \label{counter}

In this section we sketch a one dimensional example 
in which the dissipation measure $m$ does not vanish.
We assume that the zero level set of the Hamiltonian $H : \mathbb{T} \times \mathbb{R} \to \mathbb{R}$
is the smooth curve in Figure \ref{fig}, and that everywhere else in the plane 
$(x,p)$ the signs of $H$ are as shown in the picture.
In addition, $H$ can be constructed in such a way that $(D_x H, D_p H) \neq (0,0)$ 
for every $(x,p) \in \{ (x,p) \in \mathbb{T} \times \mathbb{R}: H (x,p) = 0 \}$.
That is, the zero level set of $H$ does not contain any equilibrium point.
Consider now the piecewise continuous function 
$g : [0,1] \to \mathbb{R}$, with $g(0) = g(1)$, as shown in Figure \ref{fig2}.
\begin{figure}[htbp]
\begin{minipage}[b]{6.0cm}
   \centering
  \psfrag{a}{$x$}
\psfrag{b}{$p$}
\psfrag{c}{$H(x,p) >0$}
\psfrag{d}{$H(x,p) <0$}
\psfrag{1}{$1$}
\psfrag{0}{$0$}
\includegraphics[width=2.2in]{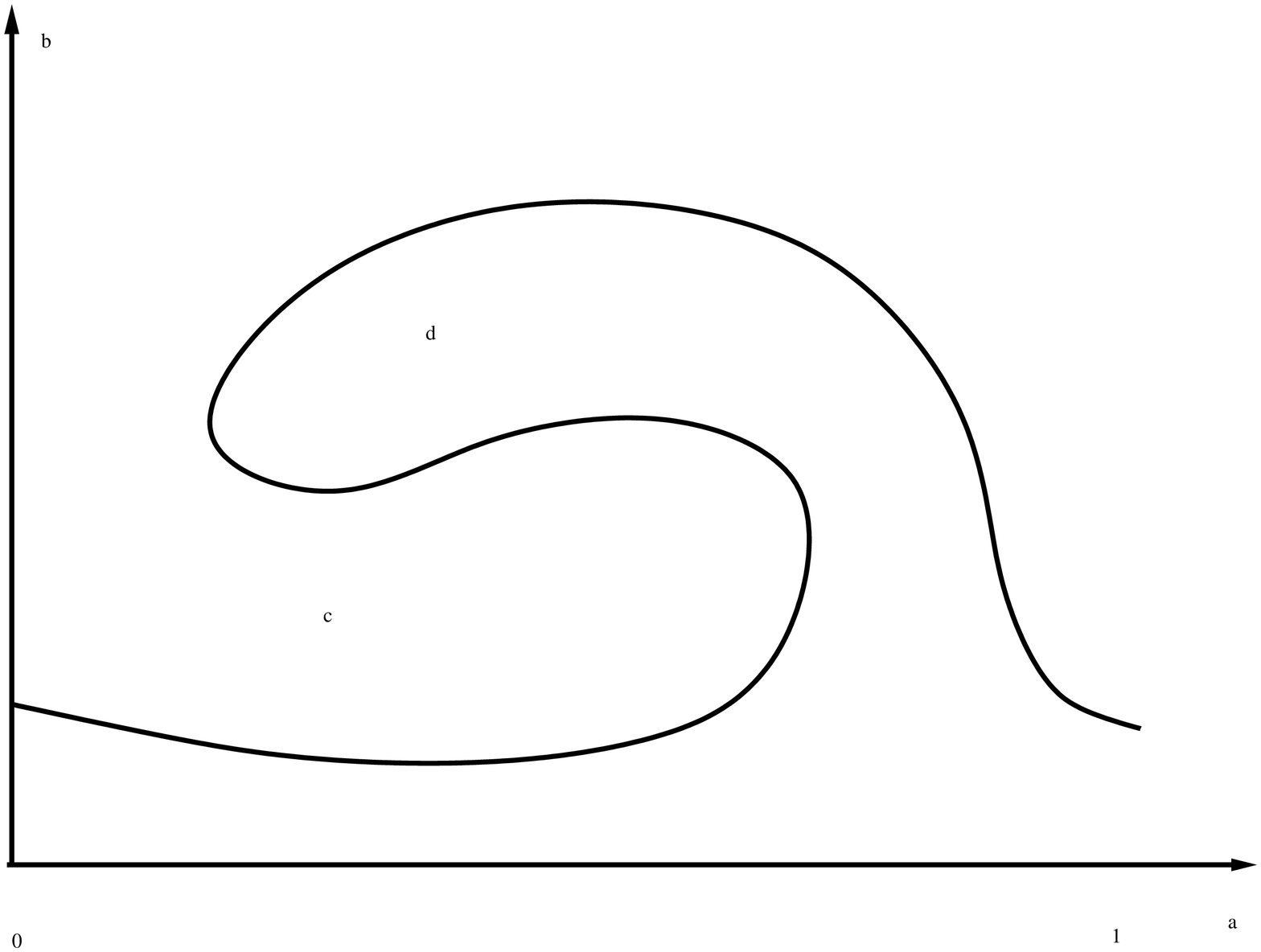}
\caption{$\{ H (x,p) = 0 \}$.}
\label{fig}
\end{minipage}
\hspace{3mm} 
\begin{minipage}[b]{6.cm}
  \centering
 \psfrag{a}{$x$}
\psfrag{b}{$p$}
\psfrag{c}{$H(x,p) >0$}
\psfrag{d}{$H(x,p) <0$}
\psfrag{1}{$1$}
\psfrag{0}{$0$}
\psfrag{g}{$g(x)$}
\psfrag{H}{$\{ H(x,p) = 0 \}$}
\includegraphics[width=2.2in]{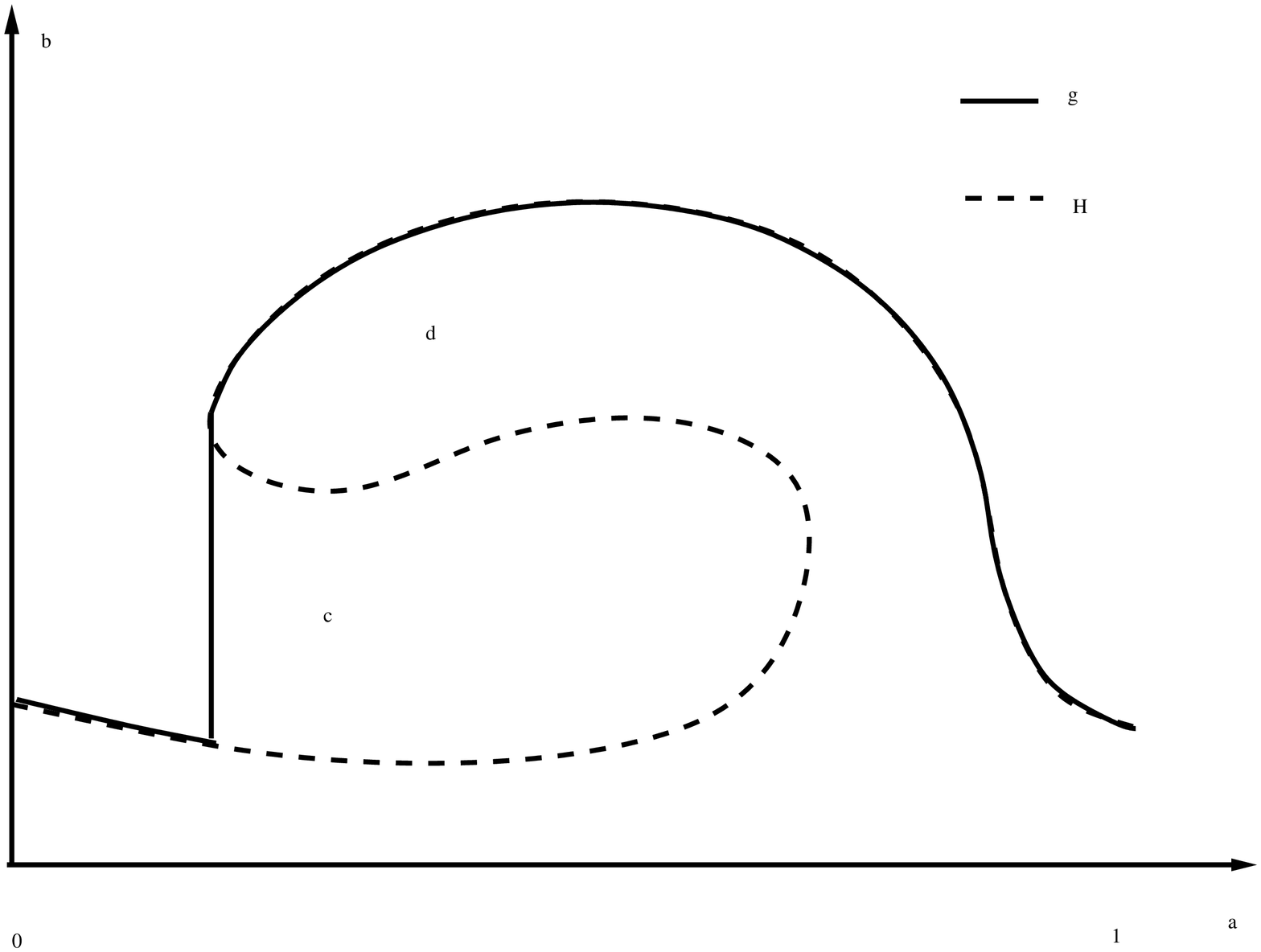}
\caption{$g(x)$.}
\label{fig2}
\end{minipage}
\end{figure}
Then, set
$$
P : = \int_0^1 g (x) \, dx,
$$
and define
$$
u (x,P) : = - P x + \int_0^x g (y) \, dy.
$$
One can see that $u(\cdot,P)$ is the unique periodic viscosity solution of
$$
H (x, P + D_x u(x,P)) = 0,
$$
that is equation \eqref{statHJ} with $\overline{H} (P) = 0$.
Assume now that a Mather measure $\mu$ exists, satisfying  property (1).
Then, the support of $\mu$ has necessarily to be concentrated on the graph of $g$,
and not on the whole level set $\{ H = 0 \}$.
However, any invariant measure by the Hamiltonian flow will be supported
on the whole set $\{ H = 0 \}$, due to the non existence of equilibria 
and to the one-dimensional nature of the problem, 
thus giving a contradiction.

\end{section}

\begin{section}{Acknowledgments}

The authors are grateful to Craig Evans and 
Fraydoun Rezakhanlou for very useful discussions on the subject of the paper.  

\end{section}

\bigskip


\bibliographystyle{alpha}

\end{document}